\documentclass{easychair}

\usepackage{cutelmet}
\usepackage{proof}
\usepackage{bussproofs}
\usepackage{latexsym}
\newtheorem{theorem}{Theorem}[section]

\newtheorem{proposition}{Proposition}[section]
\newtheorem{lemma}{Lemma}[section]
\newtheorem{examp}{Example}[section]
\newtheorem{example}{Example}[section]
                   
\def\@ptdef[#1]{\begin{definition}[#1]\rm}
\newtheorem{definition}{Definition}[section]
                     {{\hfill \Endmark}\end{definition}\@endparenv}

\def\@ptexerc[#1]{\begin{exercise}[#1]\rm}
\newtheorem{exercise}{Exercise}[section]
                     {\end{exercise}\@endparenv}

\title{Herbrand's Theorem in Refutation Schemata}
\titlerunning{Herbrand's Theorem in Refutation Schemata}

\author{
Alexander Leitsch\inst{1}
\and
    Anela Lolic\inst{2}\thanks{Partially supported by FWF project I-5848-N.}
}

\institute{
  Institute of Logic and Computation,\\TU Wien, 
   Vienna, Austria\\
  \email{leitsch@logic.at}
\and
   Kurt G\"odel Society, 
   Institute of Logic and Computation,\\TU Wien, 
   Vienna, Austria\\
   \email{anela@logic.at}
 }
 
\authorrunning{Alexander Leitsch, Anela Lolic}

\begin{document}

\maketitle

\begin{abstract}
An inductive proof can be represented as a proof schema, i.e. as a parameterized sequence of proofs defined in a primitive recursive way. A corresponding cut-elimination method, called schematic CERES, can be used to analyze these proofs, and to extract their (schematic) Herbrand sequents, even though Herbrand's theorem in general does not hold for proofs with induction inferences.
This work focuses on the most crucial part of the schematic cut-elimination method, which is to construct a refutation of a schematic formula that represents the cut-structure of the original proof schema.
We develop a new framework for schematic substitutions and define a unification algorithm for resolution schemata. Moreover, we introduce a new calculus for the refutation of formula schemata  which is simpler and more expressive than previous formalisms. Finally, we show that this new formalism allows the extraction of a structure from the refutation schema, called a Herbrand system, which represents its Herbrand sequent.
\end{abstract}

\section{Introduction}

Herbrand's theorem is one of the most important results of mathematical logic. It expresses the fact that in a formal (cut-free) proof of a prenex form propositional and quantifier inferences can be separated. In the formalism of sequent calculus this means that a so-called Herbrand sequent can be extracted from a proof; the propositional inferences operate above the sequent, the quantifier inferences below. In automated theorem proving Herbrand's theorem is used as a tool to prove completeness of refinements of resolution; moreover the theorem can yield a compact representations of proofs by abstracting away the propositional inferences. Such compact representations play also a major role in computational proof analysis: formal proofs obtained by cut-elimination from formalized mathematical proofs are typically very long covering up the mathematical content of the proofs. Experiments with the system CERES (cut-elimination by resolution, see~\cite{BL00} and~\cite{BL11}) have revealed that Herbrand forms display the main mathematical arguments of a proof in a natural way~\cite{hetzl2008herbrand}.

As most interesting mathematical proofs contain applications of mathematical induction, an extension of CERES to the analysis of inductive proofs was of major importance to turn the method into a practically useful tool for (interactive) proof analysis. A first thorough analysis of an inductive (schematic) CERES method can be found in~\cite{LPW17}; the inductive proofs investigated in this paper are those representable by a single parameter - in the formalisation by a {\em proof schema}. Here also the first concept of a {\em Herbrand system} was developed; it is essentially an extension of Herbrand's theorem from single proofs to a (recursively defined) infinite sequence of proofs. This definition of a Herbrand system represented the first step to extend Herbrand's theorem to inductive proofs (note that single proofs using the induction rule do not admit a construction of Herbrand sequents). In~\cite{Thesis.Lolic.2020} the approach in~\cite{LPW17} was extended to arbitrary many induction parameters thus considerably increasing the strength of the method. 

The schematic CERES method is capable of performing cut-elimination in presence of induction. There are other approaches to inductive inference where cut-elimination is possible as well. We just mention the works of Brotherston and Simpson~\cite{Brotherston.2005}~\cite{brotherston2011sequent} and of McDowell and Miller~\cite{mcdowell2000cut}. However, though these approaches present inductive calculi with corresponding cut-elimination methods they do not allow the construction of any Herbrand structures (in particular of Herbrand systems) as described in~\cite{LPW17} and~\cite{Thesis.Lolic.2020}.

The core of the first-order CERES method defined in~\cite{BL00} consists of the construction of a resolution refutation of a characteristic clause set (which represents the derivations of the cut formulas); this resolution refutation is then combined with a so-called proof projection to a proof of the given theorem containing only atomic cuts (called a CERES normal form). As Herbrand's theorem holds also in presence of any quantifier-free cuts, a Herbrand sequent could be extracted from such a CERES normal form. In schematic CERES the characteristic clause set becomes a characteristic formula schema (first defined in~\cite{LPW17}), projections become projection schemata. In~\cite{Thesis.Lolic.2020} it was shown (for CERES with arbitrary many induction parameters) that the construction of a proof normal  form schema, the generalization of a CERES normal form, is not necessary in constructing a Herbrand sequent: by combining the Herbrand system of the projection schema and the Herbrand system of the refutation schema a Herbrand system of the normal form schema can be constructed - without constructing the normal form schema itself. Hereby the most complex task consists in the computation of the Herbrand system of the refutation schema, which justifies the investigation of Herbrand systems in refutation schemata on its own.

In~\cite{CLL.2021} the calculus for refuting formula schemata as defined in~\cite{Thesis.Lolic.2020} was substantially extended by the use of point transition systems. The expressivity of the new calculus was demonstrated on a formula schema which could not be refuted by earlier methods. However, the construction of Herbrand systems was not addressed in~\cite{CLL.2021}, mainly due to a missing formalism for handling schematic substitutions and schematic unifications.

In this paper we develop a new framework for schematic substitutions and define a unfication algorithm for resolution schemata. This new framework allows the construction of Herbrand systems from a resolution refutation schema, one of the corner stones of the schematic CERES method; what is still missing is the combination with a Herbrand system of the projection schema as only then we obtain a Herbrand system of the normalform schema of the input proof. However, the extraction of Herbrand systems from refutation schemata is of importance on its own: indeed, the Herbrand system of a refutation schema $R$ may reveal crucial mathematical information contained in $R$. As an example the analysis of F\"urstenberg's proof of the infinitude of primes in~\cite{BHLRS.2008} (still not formalized but carried out on the mathematical meta-level) could be mentioned: here the schematic Herbrand instances  represent Euclid's construction of primes. 

To sum up we believe that a powerful formalism for the construction of Herbrand systems will be crucial in future developments of inductive proof analysis. The construction of Herbrand systems from refutation schemata as defined in this paper  can be considered a major step in this direction.

\section{Term Schemata and Formula Schemata}

Term schemata are syntactic objects describing infinite sets of terms via parameters, i.e. variables ranging over natural numbers. These syntactic objects are based on inductive definitions. In this section we will present the classes of numeric terms and of schematic individual terms. Given an parameter assignment mapping parameters to numerals, the objects of the first class will evaluate to numerals, these of the second class to first-order terms.

We start with simple numeric terms based on the constant symbol $\bar{0}$,  and the function symbols successor $s$ and predecessor $p$. Let $\Ncal$ be a countably infinite set of variables which we call {\em parameters}.
\begin{definition}\label{def.basicterm}
The set of basic terms, denoted by $T_0$ is defined inductively below:\\
$\bar{0} \in T_0$, $\Ncal \IN T_0$,\\
if $t \in T_0$ then $s(t) \in T_0$ and $p(t) \in T_0$.
\end{definition}
\begin{definition}[numerals]\label{def.numeral}
Let $T^G_0$ be the set of all ground terms (variable-free terms) in $T_0$. The subset of $T^G_0$ in which only the symbols $\bar{0}$ and $s$ occur is called the set of numerals and is denoted by $\Num$.
\end{definition}
\begin{definition}[parameter assignment]\label{def.parass}
A function $\sigma\colon \Ncal \to \Num$ is called a {\em parameter assignment}. The set of all parameter assignments is denoted by $\Pas$.	Parameter assignments $\sigma$ can be extended to functions in $T_0$ in an obvious way:
\begin{itemize}
\item $\sigma(\bar{0}) = \bar{0}$, $\sigma(s(t)) = s(\sigma(t))\Eval$, $\sigma(p(t)) = p(\sigma(t))\Eval$, 
where $p(\bar{0})\Eval = \bar{0}\Eval = \bar{0}, p(s(t))\Eval = t\Eval, s(t)\Eval = s(t\Eval)$.
\end{itemize}
Note that, for any term $t \in T_0$ and any $\sigma \in \Pas$, $\sigma(t)$ is of the form $s^\alpha(\bar{0})$ for a natural number $\alpha$.
\end{definition}

We extend $T_0$ by terms allowing primitive recursive definitions. The extension is based on the introduction of a class of function symbols and corresponding defining equations. 
\begin{definition}\label{def.numfs}
For every $n \geq 1$ we define a countably infinite class $\Fomega_n$ containing function symbols of arity $n$ (the  numeric function symbols of arity $n$) and $\Fomega = \Union^\infty_{n=1} \Fomega_n$ (the numeric function symbols).
For the basic functions $s$ and $p$ we require  $s,p \not \in\Fomega$.
\end{definition}
Below we extend the set of terms $T_0$  to the set of {\em numeric terms} $\Tomega$:
\begin{definition}[$\Tomega$]\label{def.num_terms}
The set $\Tomega$ is defined inductively as follows:\\
$0 \in \Tomega$, $\Ncal \IN \Tomega$,\\
if $t \in \Tomega$ then $s(t) \in \Tomega$ and $p(t) \in \Tomega$,\\
if $t_1,\ldots,t_m \in \Tomega$ and $f \in \Fomega_m$ then $f(t_1,\ldots,t_m) \in \Tomega$.
\end{definition}
An {\em $\omega$-theory} provides the recursive definitions corresponding to the numeric function symbols.
\begin{definition}[$\omega$-theory]\label{def.theory}
A tupel $T\colon (f,\Fcal,D(\Fcal),<)$ is called an $\omega$-theory of $f$ if \\
$f \in \Fcal$, $\Fcal$ is a finite subset of $\Fhat$,
$<$ is an irreflexive, transitive relation on $\Fcal \union \{s,p\}$ such that for all $g \in \Fcal\setminus \{f\}\colon g<f$ ($f$ is 
      maximal in $\Fcal$) and for all $g \in \Fcal$ $s < g$ and $p < g$ ($s$ and $p$ are minimal).\\[1ex]
$D(\Fcal) = \Union_{g \in \Fcal}D(g)$ where $D(g)$ is a set of two equations (we assume $g \in \Fomega_{m+1}$) of the form:
$$\{g(x_1,\ldots,x_m,s(y)) = t^g_S\{z \ass g(x_1,\ldots,x_m,y)\},\ g(x_1,\ldots,x_m,\bar{0}) = t^g_B \}$$
where $t^g_S,t^g_B \in \Tomega$, $\Ncal(t^g_S) \IN \{x_1,\ldots,x_m,y,z\}$, $\Ncal(t^g_B) \IN \{x_1,\ldots,x_m\}$, $\Gcal\colon \Fomega(t_S) \union \Fomega(t_B) \IN \Fcal$,  and  for all $h \in \Gcal: h<g$.
\end{definition}
A semantics of terms in $T^\omega$ and thus of $\omega$-theories $T$ is defined in~\cite{CLL.2021}. Intuitively the evaluation of a term  $g(x_1,\ldots,x_m,y)$ (with $D(g)$ as above) under a parameter assignment $\sigma$ corresponds to the computation of the recursive program 
$$g(x_1,\ldots,x_m,y) = \If\ y=0\ \Then\ t^g_B\ \Else\ t^g_S\{z \ass g(x_1,\ldots,x_m,y-1)\}$$
on the input $\sigma(x_1),\ldots,\sigma(x_m),\sigma(y)$. The numeral obtained by evaluation of a term $t$ in $T$ under a parameter assignment $\sigma$ is denoted by $\sigma(t)\Eval_T$. An $\omega$-theory is nothing else than a theory for computing primitive recursive functions. \\[1ex]
The next term schemata we define are {\em individual term schemata}.
\begin{example}\label{ex.termschema}
Let $f$ be a one-place function symbol and $x$ be an individual (first-order) variable. Then the terms $x,f(x),f(f(x)),\ldots$ form an infinite sequence which, in an informal mathematical notation, could be denoted by $f^n(x)_{n \in \N}$. In analogy to numeric terms we describe this sequence recursively via a new (defined) function symbol; let us call this symbol $\fhat$ and define it as 
$$\fhat(x,\bar{0}) = x,\ \fhat(x,s(n)) = f(\fhat(x,n)).$$
\end{example} 
\begin{definition}[function symbols]\label{def.FSiota}
For every $n \geq 0$ we define an infinite set of function symbols of arity $n$ (ranging over individual domains) which we denote by $\Fiota_n$, the set of all function symbols is defined as $\Fiota = \Union_{n \in \N}\Fiota_n$.
\end{definition}
We extend the usual concept of variables by so-called variable expressions which will be needed in formula schemata and refutation schemata. 
\begin{definition}[individual variables]\label{def.Viota}
We define as $\Viota_0$ an infinite set of first-order variables as usual; for variables in $\Viota_0$ we generally use the symbols $x,y,z,x_1,\ldots$. For $k \geq 1$ we define $\Viota_k$ as the set of all $k$-ary variable classes, where (for $k \geq 1$) an $k$-ary variable class is an infinite set of variables $X$ together with a bijective mapping $\iota_X\colon \Num^k \to X$; for the element $\iota_X(\bar{i}_1,\ldots,\bar{i}_k)$ in $X$ we simply write 
$X(\bar{i}_1,\ldots,\bar{i}_k)$ and call this element a {\em variable} in $X$. Variables $X(t_1,\ldots,t_k)$ and $Y(s_1,\ldots,s_l)$ are defined to be different if either $X \neq Y$ or $X=Y$, $l=k$ and $s_i \neq t_i$ for some $i$ in $\{1,\ldots,k\}$.	All $n$-ary variable classes are mutually disjoint.  The set $\Viota$ is defined as $\Union_{n \in \N}\Viota_n$.\\[1ex]
Let $X \in \Viota_k$ and $t_1,\ldots,t_k$ be terms in $T_0$; then the expression $X(t_1,\ldots,t_k)$ is called a {\em variable expression}. Note that for $t_1,\ldots,t_k \in \Num$ $X(t_1,\ldots,t_k)$ is a variable in $X$. The set of all variables (in some $X$) is denoted by $\Variota$. 
\end{definition}
Variable expressions of the form $X(t_1,\ldots,t_k)$ are crucial in the definition of formula schemata, enabling a schematic number of different variables in formulas. $X$ could be interpreted as a second order variable of type $\omega^k \to \iota$. However, we prefer the term {\em variable class} as we neither quantify over $X$ nor apply a second-order substitution on $X$. So $X$ is in some sense a {\em passive} second-order variable. In fact we will only substitute variable expressions by schematic terms, not the variable classes $X$ themselves.
\begin{definition}[individual terms]\label{def.Tiota}
The set $\Tiota$ of individual terms is defined inductively as follows.\\
$\Fiota_0 \IN \Tiota$,\\
all variable expressions are in $\Tiota$,\\
if $t_1,\ldots,t_k \in \Tiota$ and $f \in \Fiota_k$ then $f(t_1,\ldots,t_k) \in \Tiota$.\\
The subset of $\Tiota$ containing only variable expressions which are variables is denoted by $\Tiota_0$.
\end{definition}
In the definition of schematic individual terms we use a tupel of variables and a second one of parameters. The lengths of the two tuples are expressed in the concept of profile, where a {\em profile} is just a triple  $(i,j)$ for $i,j \in \N^+$.
\begin{definition}[schematic term symbols]\label{def.schem-t-symbols}
To each profile $\pi$ we assign an infinite set of symbols $\Fhat(\pi)$ which we call {\em term symbols}. The set obtained by taking a union over all $\Fhat(\pi)$ for profiles $\pi$ is denoted by $\Fhat$.
\end{definition}
Term symbols take over the role the symbols in $\Fomega$ played in the definition of $\Tomega$. 
\begin{definition}[schematic individual terms]\label{def.schem-ind-terms}
We define the set $\Tiotahat$ (the set of {\em schematic individual terms}) inductively as follows:\\
$\Tiota \IN \Tiotahat$.\\
If $t_1,\ldots,t_k \in \Tiotahat$ and $f \in \Fiota_k$ then $f(t_1,\ldots,t_k) \in \Tiotahat$.\\
If $\pi = (i,j)$ is a profile and $\that \in \Fhat(\pi)$, $t_l \in \Tiotahat$ for $l=1,\ldots,i$, and $s_1,\ldots,s_k \in \Tomega$ then 
$\that(t_1,\ldots,t_i,s_1,\ldots,s_j) \in \Tiotahat$.
\end{definition}
It remains to define corresponding equations to the symbols in $\Fhat$. So we have to define an $\iota$-theory for the symbols in 
$\Fhat(\pi)$ together with an $\omega$-theory for the symbols in $\Fomega$. We write $\Fhat(T)$ for the set of symbols in $\Fhat$ appearing in a set of terms $T \IN \Tiotahat$. 
\begin{definition}[$\omega\iota$-theory]\label{def.iotatheory}
A tupel $(T,T')$ is called an $\omega\iota$-theory of $\that$ if \\
$T = (\that,\Fcal,D(\Fcal),V,<)$, $T'\colon (\fhat,\Fcal',D(\Fcal'),<')$ is an $\omega$-theory.\\
$\that \in \Fcal$, $\Fcal$ is a finite subset of $\Fhat$, $V$ is a finite set of variable expressions.\\
$<$ is an irreflexive, transitive relation on $\Fcal \union \Tiota$ such that for all $\shat \in \Fcal\setminus \{\that\}\colon \shat<\that$ ($\that$ is maximal in $\Fcal$) and for all $\shat \in \Fcal$ and $t \in \Tiota$ we have $t < \shat$ (the terms in $\Tiota$ are smaller than all schematic term symbols).\\
$D(\Fcal) = \Union_{\shat \in \Fcal}D(\shat)$ where $D(\shat)$ is a set of two equations.  We assume that 
$\pi(\shat) = (i,j)$, $x_l \in \Viota_0$ for $l =1,\ldots,i$ and $n_1,\ldots,n_{j-1},n_j \in \Ncal$. 
Then $D(\shat)$ is defined as 
\[
\begin{array}{l}
\{\shat(x_1,\ldots,x_i,n_1,\ldots,n_{j-1},s(n_j)) = \shat_S\{z \ass \shat(x_1,\ldots,x_i,n_1,\ldots,n_{j-1},n_j)\},\\
 \shat(x_1,\ldots,x_i,n_1,\ldots,n_{j-1},\bar{0}) =\shat_B\}
\end{array}
\]
Here we require $\shat_S, \shat_B \in \Tiotahat(T,T')$, where $\Tiotahat(T,T')$ is the subset of all terms $t$  in $\Tiotahat$ such that 
$\Fhat(t) \IN \Fcal$ and $\Fomega(t) \IN \Fcal'$, \\
$\Viota(\shat_S) \IN \{x_1,\ldots,x_i,z\}$, $\Ncal(\shat_S) = \{n_1,\ldots,n_j\}$, and $z \in \Viota_0 \setminus \{x_1,\ldots,x_i\}$;
\\ 
$\Viota(\shat_B) \IN \{x_1,\ldots,x_i\}$, $\Ncal(\shat_B) \IN \{n_1,\ldots,n_{j-1}\}$, \\
Let $\Gcal = \Fhat(\shat_S) \union \Fhat(\shat_B)$; then,  for all $t \in \Gcal: t<\shat$.
\end{definition}
The semantics of terms in $\Tiota$ and $\Tiotahat$ is defined in~\cite{CLL.2021}. In contrast to an $\omega$-theory where, under parameter assignments, terms evaluate to numerals, terms in $\Tiotahat$ evaluate to ordinary first-order terms.\\[1ex]
In the next step we define {\em schematic formulas}. 
 Let $\Po_k$ be the set of of $k$-ary predicate symbols for every $k \geq 0$, where $\Po_0 = \{\top,\bot\}$ and $\Po_k$ is infinite for every $k \geq 1$.  
In contrast to term schemata we will make use of variable classes $\Viota_k$ for $k>0$ in the definition of recursive predicates. 
We introduce schematic predicate symbols via profiles, where predicate profiles are now a three tupel of numbers $(i,j,k)$ for $j,k>0$. 
\begin{definition}[schematic predicate symbols]\label{def.schem-p-symbols}
To each predicate profile $\pi$ we assign an infinite set of symbols $\Phat(\pi)$ which we call {\em schematic predicate symbols}. The set obtained by taking a union over all $\Phat(\pi)$ for profiles $\pi$ is denoted by $\Phat$.
\end{definition}
For defining recursions for schematic predicate symbols we need formula variables; the set of all formula variables is defined as $\FV$ and as a name for elements in $\FV$ we use $\xi,\xi_1,\ldots$.
\begin{definition}[basic formulas]\label{def.Fbasic}
Let $T^*\colon (T,T')$ be an $\omega\iota$-theory. The set $\Fob(T^*)$ of basic formulas in $T^*$ is defined inductively as follows.\\
$\FV \IN \Fob(T^*)$.\\
If $P \in \Po_k$ and $t_1,\ldots,t_k \in \Tiotahat(T^*)$ then $P(t_1,\ldots,t_k) \in \Fob(T^*)$.\\
If $F \in \Fob(T^*)$ then $\neg F \in \Fob(T^*)$.\\
If $F_1,F_2 \in \Fob(T^*)$ then $F_1 \land F_2, F_1 \lor F_2 \in \Fob(T^*)$.\\
 For the subset of $\Fob$ containing no formula variables we write $\Fobminus$.
\end{definition}
Below we define the class of quantifier-free schematic formulas which - for simplicity - we call schematic formulas.
\begin{definition}[schematic formulas]\label{def.schem-formulas}
We define the set $\Fhat$ (the set of {\em schematic formulas}) inductively as follows:\\
$\FV \IN \Fhat$.\\
If $P \in \Po_k$ and $t_1,\ldots,t_k \in \Tiotahat$ then $P(t_1,\ldots,t_k) \in \Fhat$.\\
If $\pi = (i,j,k)$ is a profile and $\phat \in \Phat(\pi)$, $X_l \in \Viota_i$ for $l=1,\ldots,j$, $t_1,\ldots,t_k \in \Tomega$ then 
$\phat(X_1,\ldots,X_j,t_1,\ldots,t_k) \in \Fhat$.\\
If $F \in \Fhat$ then $\neg F \in \Fhat$.
If $F_1,F_2 \in \Fhat$ then $F_1 \land F_2, F_1 \lor F_2 \in \Fhat$.
\end{definition}
The following concept of an $\omega\iota o$-theory is based on an extension of an $\omega\iota$-theory. 
\begin{definition}[$\omega\iota o$-theory]\label{def.omegaio-theory}
A tupel $T^+\colon (T,T',T'')$ is called an $\omega\iota o$-theory of $\phat$ if \\
$T = (\phat,\Pcal,D(\Pcal),V,<_p)$,\\
$(T',T'')$ is an $\omega\iota$-theory,\\
$T' = (\that,\Fcal,D(\Fcal),V',<)$, $T''\colon (\fhat,\Fcal',D(\Fcal'),<')$, \\
$\phat \in \Pcal$ and $\Pcal$ is a finite subset of $\Phat$,\\
$V$ is a finite subset of $\Viota_k$ for some $k \in \N$,\\
$<$ is an irreflexive, transitive relation on $\Pcal$ such that for all $\qhat \in \Pcal\setminus \{\phat\}\colon \qhat<\phat$ ($\phat$ is maximal in $\Pcal$). We extend $<$ to schematic formulas as follows: if $\qhat < \rhat$ then (for $\qhat \in \Phat(i,j,k)$ and $\rhat \in \Phat(i',j',k')$) then 
$$\qhat(X_1,\ldots,X_j,t_1,\ldots,t_k) < \rhat(X'_1,\ldots,X'_{j'},t'_1,\ldots,t'_{k'})$$  
for all $t_\alpha, t'_\beta$.\\
And for all schematic formulas $F'$ containing a symbol in $\Pcal$ and for all $F \in \Fob$ we have 
$F < F'$,  i.e. the formulas in $\Fob$ are smaller than all formulas containing schematic predicate symbols.\\[1ex]
$D(\Pcal) = \Union_{\qhat \in \Pcal}D(\qhat)$ where $D(\qhat)$ is a set of two equations.  We assume that 
$\pi(\qhat) = (k,i,j)$, $X_l \in V_k$ for $l =1,\ldots,i$ and $n_1,\ldots,n_j \in \Ncal$.\\ 
Then $D(\qhat)=$
\[
\begin{array}{l}
\{\qhat(X_1,\ldots,X_i,n_1,\ldots,n_{j-1},s(n_j)) = \qhat_S\{\xi \ass \qhat(X_1,\ldots,X_i,n_1,\ldots,n_{j-1},n_j)\},\\ 
\qhat(X_1,\ldots,X_i,n_1,\ldots,n_{j-1},\bar{0}) = \qhat_B \}.
\end{array}
\]
where we require $\qhat_S, \qhat_B \in \Fhat(T^+)$, where $\Fhat(T^+)$ is the subset of all formula schemata $F$ in $\Fhat$ such that 
$\Phat(F) \in \Pcal$, $\Fhat(F) \IN \Fcal$ and $\Fomega(F) \IN \Fcal'$, \\
$\Viota(\qhat_S) \IN \{X_1,\ldots,X_i\}$, $\FV(\qhat_S) \IN \{\xi\}$, $\Ncal(\qhat_S) = \{n_1,\ldots,n_j\}$; \\
$\Viota(\qhat_B) \IN \{X_1,\ldots,X_i\}$, $\FV(\qhat_B) = \emptyset$, $\Ncal(\qhat_B) \IN \{n_1,\ldots,n_{j-1}\}$.\\[1ex]
Let $\Gcal = \Phat(\qhat_S) \union \Phat(\qhat_B)$; then,  for all $r \in \Gcal: r <\qhat$.
\end{definition}
The semantics of $\omega \iota o$-theories is defined in~\cite{CLL.2021}. If $T^+$ is an $\omega \iota o$-theory, $F$ a formula in $T^+$,  and $\sigma$ is a parameter assignment, the evaluation of $F$ under $\sigma$ is denoted by $\sigma(F)\Eval_{T^+}$; $\sigma(F)\Eval_{T^+}$ is an ordinary quantifier-free first order formula. 
\begin{definition}\label{def.unsat}
Let $F$ be a formula in an $\omega \iota o$-theory $T^+$. Then $F$ is {\em unsatisfiable} if for all parameter assigments $\sigma$ the formula $\sigma(F)\Eval_{T^+}$ is unsatisfiable.
\end{definition}
For illustration we give an example of such a theory and an evaluation of a schematic formula under a parameter assignment.
\begin{examp}\label{ex.omegaiotao}
Let $T^+\colon (T,T',\emptyset)$ a triple such that 
\begin{eqnarray*}
T &=& (\phat,\{\phat\},D(\{\phat\}),\{X,Y\},\emptyset),\\
T' &=& (\fhat,\{\fhat\},D(\{\fhat\}),\{y\},\emptyset),\\
D(\phat) &=& \{\phat(X,s(n)) = Q(\fhat(X(n),n),Y(n)) \lor \phat(X,n),\ \phat(X,\bar{0}) = \neg P(X(\bar{0}))\},\\
D(\fhat) &=& \{\fhat(y,s(n)) = f(\fhat(y,n)),\ \fhat(y,\bar{0})=y\}.
\end{eqnarray*}
Then $T^+$ is an $\omega\iota o$-theory. Consider the formula $\phat(X,n)$ in $T^+$ and $\sigma(n) = \bar{3}$.  Then 
$$\sigma(\phat(X,n))\Eval_{T^+} = Q(f(f(X(\bar{2}))),Y(\bar{2})) \lor Q(f(X(\bar{1})),Y(\bar{1})) \lor Q(X(\bar{0}),Y(\bar{0})) \lor \neg P(X(\bar{0})).$$
The schematic formula $\phat(X,n)$ is satisfiable as, for $\sigma(n) = \bar{3}$, $\sigma(\phat(X,n))\Eval_{T^+}$ is satisfiable. 
\end{examp}
In Example~\ref{ex.omegaiotao} we have terms of the form $\fhat(X(n),n)$ which, due to the use of variable expressions, are not covered by Definition~\ref{def.schem-ind-terms}. This extension of the syntax of schematic terms is subject of the definition below.
\begin{definition}\label{def.ext-schem-terms}
Let $t$ be a term in $\Tiotahat(T^*)$ for a $\omega\iota$-theory $T^*$ and $\Viota(t) = \{x_1,\ldots,x_\alpha\}$; let $V_1,\ldots,V_\alpha$ be variable expressions and 
$\theta = \{x_1 \ass V_1,\ldots,x_\alpha \ass V_\alpha\}$. Then the expression $t\theta$ is called an {\em extended schematic term}. 
\end{definition}
Note that, due to the use of variable classes in the definition of $\omega\iota o$-theories, all terms occurring in formula schemata of a nontrivial theory (i.e. we have recursive definition of predicate symbols) are extended schematic terms.

\section{Schematic Substitution and Unification}\label{sec.sunification}

\subsection{Schematic Substitution}

Similar to terms and schematic terms {\em schematic substitutions} are expressions which evaluate to substitutions under parameter assignments. Below we define the syntactic expressions which will turn out useful in schematic refutational calculi.
\begin{definition}\label{def.s-substitution}
Let $T^*\colon (T_1,T_2)$ be an $\omega \iota$-theory. A mapping $\Theta\colon \Viota_0 \to \Tiotahat(T^*)$ is called an {\em s-substitution} over $T^*$ if $\Theta(x) \neq x$ only for finitely many $x \in \Viota_0$. We define the {\em domain} of $\Theta$ as 
\begin{eqnarray*}
\dom(\Theta) &=& \{x \mid x \in \Viota_0, \Theta(x) \neq x\} \mbox{ and the range of }\Theta \mbox{ as}\\
\rg(\Theta) &=& \{\Theta(x) \mid x \in \dom(\Theta)\}.
\end{eqnarray*}
If $\dom(\Theta) = \{x_1,\ldots,x_\alpha\}$ and $\Theta(x_i) = t_i$ for $i=1,\ldots,\alpha$ then we write $\Theta$ as $\{x_1 \ass t_1,\ldots, x_\alpha \ass t_\alpha\}$.
\end{definition}
\begin{examp}\label{ex.s-substitution}
Let $T^*\colon (T_1,T_2)$ be an $\omega \iota$-theory 
where $T_1 =  (\that,\{\that,\shat\},D(\{\that,\shat\}),\{x,y\},<)$, $\shat < \that$ and 
\begin{eqnarray*}
D(\that) &=& \{\that(x,y,n_1,s(n_2)) = g(\that(x,y,n_1,n_2),\shat(x,n_1)), \that(x,y,n_1,\bar{0}) = y\},\\
D(\shat) &=& \{\shat(x,s(n_1)) = h(\shat(x,n_1)),\ \shat(x,\bar{0}) = x\}.
\end{eqnarray*}
Let $u,v,x_1,x_2 \in \Viota_0$. Then we define 
$$\Theta = \{u \ass h(\that(x_1,x_2,n_1,n_2)),\ v \ass \shat(g(x_1,x_2),n_1)\}$$
$\Theta$ is an s-substitution over $T^*$ where 
$$\dom(\Theta) = \{u,v\},\ \rg(\Theta) = \{ h(\that(x_1,x_2,n_1,n_2)), \shat(g(x_1,x_2),n_1)\}.$$
\end{examp}
In the same way as schematic terms $t$ evaluate to first-order terms $\sigma(t)\EvalTstar$ under a parameter assignment $\sigma$, $s$-substitutions evaluate to ordinary first-order substitutions.
\begin{definition}\label{def.eval-s-subst}
Let $T^*\colon (T_1,T_2)$ be an $\omega \iota$-theory and $\Theta\colon \{x_1 \ass t_1,\ldots,x_\alpha \ass t_\alpha\}$ be an $s$-substitution over $T^*$. Then, for $\sigma \in \Scal$,  
$$\Theta[\sigma] = \{x_1 \ass \sigma(t_1)\EvalTstar,\ldots, x_\alpha \ass \sigma(t_\alpha)\EvalTstar\}$$
is called the $\sigma$-evaluation of $\Theta$, where $\sigma(t_i)\EvalTstar$ denotes the evaluation of $t_i$ under $\sigma$ in the theory $T^*$.
\end{definition}
\begin{proposition}\label{prop.s-subst}
Let $T^*$ be an $\omega \iota$-theory, $\sigma \in \Scal$ and $\Theta$ be an $s$-substitution over $T^*$. Then 
$\Theta[\sigma]$ is a first-order substitution.
\end{proposition}
\begin{proof}
Let $\Theta = \{x_1 \ass t_1,\ldots,x_\alpha \ass t_\alpha\}$ be an $s$-substitution over $T^*$ and $\sigma \in \Scal$. By~\cite{CLL.2021} $t'_i\colon \sigma(t_i)\EvalTstar \in \Tiota_0$ for $i=1,\ldots,\alpha$. Therefore $t'_i$ are first-order terms and 
\begin{eqnarray*}
\Theta[\sigma] &=& \{x_1 \ass t'_1,\ldots,x_\alpha \ass t'_\alpha\} \mbox{ is a first-order substitution}.
\end{eqnarray*}
\end{proof}
\begin{examp}\label{ex.eval-s-subst}
We take $T^*$ and $\Theta$ from Example~\ref{ex.s-substitution} and define $\sigma$ by $\sigma(n_1) = \bar{2},\sigma(n_2) = \bar{1}$. Then 
$$\Theta[\sigma] = \{u \ass h(\that(x_1,x_2,\bar{2},\bar{1}))\EvalTstar,\ v \ass \shat(g(x_1,x_2),\bar{2})\EvalTstar\}.$$
As $\that(x_1,x_2,\bar{2},\bar{1})\EvalTstar = g(x_2,h(h(x_1)))$ and $\shat(g(x_1,x_2),\bar{2})\EvalTstar = h(h(g(x_1,x_2)))$ we obtain 
$$\Theta[\sigma] = \{u \ass h(g(x_2,h(h(x_1)))), \ v \ass h(h(g(x_1,x_2)))\}.$$

\end{examp}

Now we have to define formally how $s$-substitutions apply to schematic terms and formulas.  In particular we define $s$-substitutions on terms in $\Tiota$ and $\Tiotahat$ first. In our first step we avoid terms containing variable classes, so our variables are  ordinary variables in $\Viota_0$.  

\begin{definition}[$s$- substitutions on individual terms]\label{def.ssub-Tiota}
Let $\Theta\colon \{x_1 \ass t_1,\ldots, x_\alpha \ass t_\alpha\}$ be an $s$-substitution, then we define 
\begin{itemize}
\item $c\Theta = c$ for $c \in \Fiota_0$,
\item if $x \in \Viota_0$ and $x \not \in \{x_1,\ldots,x_\alpha\}$ then $x\Theta = x$,
\item if $x = x_i \in \{x_1,\ldots,x_\alpha\}$ then $x_i \Theta = t_i$,
\item if $f \in \Fiota_k$ then $f(t_1,\ldots,t_k)\Theta =  f(t_1\Theta,\ldots,t_k\Theta)$.
\end{itemize}
\end{definition}
In the next step we extend $s$-substitutions to term schemata.
\begin{definition}\label{def.ssub-simple-schem-terms}
Let $\Theta$ be an $s$-substitution.
\begin{itemize}
\item If $t_1,\ldots,t_k \in \Tiotahat$ and $f \in \Fiota_k$ then $f(t_1,\ldots,t_k)\Theta = f(t_1\Theta,\ldots, t_k\Theta)$,
\item if $\pi = (j,k)$ is the profile of $\that \in \Fhat(\pi)$, $s_l \in \Tiotahat_s$ for $l=1,\ldots,j$, $t_1,\ldots,t_k \in \Tomega$ then
$$\that(s_1,\ldots,s_j,t_1,\ldots,t_k)\Theta = \that(s_1\Theta,\ldots,s_j\Theta,t_1,\ldots,t_k).$$
\end{itemize}
\end{definition}
\begin{examp}\label{ex.apply-s-substitution}
Let $T^*\colon (T_1,T_2)$ be the theory in Example~\ref{ex.s-substitution}. Then 
$$t\colon h(\that(\that(x,y,n,m),\shat(x,n),n,m)) \in \Tiotahat(T^*).$$
Let $\Theta = \{x \ass \shat(u,k), y \ass g(u,v)\}$. Then 
$$t\Theta = h(\that(\that(\shat(u,k),g(u,v),n,m),\shat(\shat(u,k),n),n,m)).$$
\end{examp}
Even when the theories do not contain $n$-ary variable expressions for $n>0$ (which concerns the recursive definition of schematic terms) the use of nontrivial variable expressions cannot be avoided due to the definition of formula schemata. We give an example below. 
\begin{examp}\label{ex.ssub-form-schema2}
Let $(T,T',\emptyset)$ the $\omega\iota o$-theory defined below. 
\begin{itemize}
\item $T = (\qhat,\{\qhat\},D(\{\qhat\}),\{X\},\emptyset)$ where $X \in \Viota_1$, $\pi(\qhat) = (1,1,1)$ and
\end{itemize}
$$D(\qhat) = \{ \qhat(X,s(n)) = \qhat(X,n) \lor Q(\that(X(s(n)),s(n))),\ \qhat(X,\bar{0}) = Q(X(\bar{0}))\}.$$
Here we have $\qhat_S  = \xi \lor Q(\that(X(s(n)),s(n)))$ for $\xi \in \FV$, $\qhat_B = Q(X(\bar{0}))$. 
\begin{itemize}
\item $T'=(\that,\{\that\},D(\{\that\}),\{x\},\emptyset)$ for $x \in \Viota_0$, $\pi(\that) = (1,1)$ and 
\end{itemize}
$$D(\that) = \{\that(x,s(n)) = h(\that(x,n)),\ \that(x,\bar{0}) = x \}.$$
$\that_S = h(z)$, $\that_B = x$.\\[1ex]
The theory $T^*\colon (T,T',\emptyset)$ contains $(T',\emptyset)$ which is an $\omega \iota$-theory. However, due to the definition of $\qhat$ we obtain schematic formulas of the form $\qhat(X,n) \lor Q(\that(X(s(n)),n))$ where the term $\that(X(s(n)),s(n))$ contains a unary variable expression - which is not a variable as $n$ is a parameter. In fact, $\that(X(s(n)),s(n))$ is an {\em extended schematic term} in $T^*$ as in Definition~\ref{def.ext-schem-terms}. 
\end{examp}
\begin{definition}[extended $s$-substitution]\label{def.extssub}
Let $T^*$ be an $\omega \iota o$-theory. An extended $s$-substitution over $T^*$ is an expression of the form 
$\{V_1 \ass t_1,\ldots, V_\alpha \ass t_\alpha\}$
where, for $i=1,\ldots,\alpha$,  $V_i$  are variable expressions and $t_i$ are extended schematic terms in $T^*$, and for all $\sigma \in \Scal$ and all $i,j \in \{1,\ldots,\alpha\}$ such that $i \neq j$ we have $\sigma(V_i)\Eval \neq \sigma(V_j)\Eval$. If there exists a $\sigma \in \Scal$ such that $\sigma(V_i)\Eval = \sigma(V_j)\Eval$ then we say that $\{V_i, V_j\}$ is {\em parameter-unifiable}. 
\end{definition}
Below we define the semantics of extended $s$-substitutions and prove that they evaluate to first-order substitutions by normalization under parameter assignments.
\begin{definition}\label{def.sem-extssub}
Let $\Theta\colon \{V_1 \ass t_1,\ldots, V_\alpha \ass t_\alpha\}$ be an extended $s$-substitution over a theory $T^*$ and let  $\sigma \in \Scal$. Then we define 
$\Theta[\sigma] = \{\sigma(V_1)\Eval \ass \sigma(t_1)\Eval, \ldots, \sigma(V_\alpha)  \ass \sigma(t_\alpha)\Eval\}.$
\end{definition}
\begin{proposition}
Let $\Theta$ be an extended $s$-substitution and $\sigma \in \Scal$. Then $\Theta[\sigma]$ is a first-order substitution.
\end{proposition}
\begin{proof}
Let $\Theta = \{V_1 \ass t_1,\ldots, V_\alpha \ass t_\alpha\}$. In $\Theta[\sigma]$, $\sigma(V_i)\Eval$ are first-order variables and 
$\sigma(t_i)\Eval$ are first-order terms. Moreover, by Definition~\ref{def.extssub}, $\sigma(V_i)\Eval \neq \sigma(V_j)\Eval$ for $i \neq j$ and so all resulting first-order variables are pairwise different; hence $\Theta[\sigma]$ is a first-order substitution.
\end{proof} 
\begin{examp}
Let $\Theta = \{X(n,n) \ass g(X(\bar{0},n)),\  X(s(n),m) \ass h(X(m,m)), Y(m) \ass a\}$,
where $X,Y$ are variable classes and $g,h$ are function symbols.
$\Theta$ is an extended $s$-substitution: it is easy to see that for all $\sigma \in \Scal$ the sets $\{X(n,n),$ $X(s(n),$ $m)\}$, $\{X(n,n),$ $Y(m)\}$ and $\{X(s(n)$ $,m),$ $Y(m)\}$ are not parameter-unifiable. Let $\sigma \in \Scal$ such that $\sigma(n) = \bar{0}, \sigma(m) = \bar{1}$. Then 
$$\Theta[\sigma] = \{X(\bar{0},\bar{0}) \ass g(X(\bar{0},\bar{0})),  X(\bar{1},\bar{1}) \ass h(X(\bar{1},\bar{1})), Y(\bar{1}) \ass a\}.$$
Clearly, $\Theta[\sigma]$ is a first-order substitution which, under variable renaming, is equivalent to 
$$\{x \ass g(x),\ y \ass h(y), z \ass a\}.$$
\end{examp}
In extending schematic terms to extended schematic terms it makes sense to restrict the arguments of the variable expressions to simplify the handling of extended $s$-substitutions.  Below we define a restricted class which is large enough for schematic Ceres.
\begin{definition}\label{def.standard-term}
A term $t$ in $\Tiotahat$ is called a {\em standard term} if all for variable expressions $X(\vec{s})$ occurring in $t$ for $X \in \Viota_k$ and $k \geq 1$ the following conditions are fulfilled: 
\begin{itemize}
\item assigned to $X$ is a unique parameter list $(n_1,\ldots,n_k)$,  
\item  the only variable expressions starting with $X$ are of the form $X(r_1,\ldots,r_k)$  where each $r_i$ is of one of the forms $n_i$, 
$\bar{0}$, $p(n_i)$ or $s(n_i)$. We call these variable expressions {\em standard}.
\end{itemize}
\end{definition}
\begin{definition}\label{def.ssub-standard}
A {\em standard $s$-substitution} is an expression of the form 
$$\{X_1(r^1_1,\ldots,r^1_{\alpha_1}) \ass t_1, \ldots, X_k(r^k_1,\ldots,r^k_{\alpha_k}) \ass t_k\}$$
where
\begin{itemize}
\item  $X_1,\ldots,X_k$ are  variable classes,
\item the parameter list assigned to $X_i$ is $(n^i_1,\ldots,n^i_{\alpha_i})$ for $i \in \{1,\ldots,k\}$,
\item the $X_i(r^i_1,\ldots,r^i_{\alpha_i})$ are standard variable expressions, 
\item the $t_i$ are standard terms  in $\Tiotahat$,
\item the sets  $\{X_i(r^i_1,\ldots,r^i_{\alpha_i}), X_j(r^j_1,\ldots,r^j_{\alpha_j})\}$ for $i \neq j$ are not parameter-unifiable.
\end{itemize}
\end{definition}
\begin{examp}\label{ex.ssub-standard}
Let $X$ be a variable class with assigned parameter $n$, $Y$ be a variable class with assigned parameters $m,n$ and $t_1,t_2,t_3$ standard terms  in $\Tiotahat$. The $s$-substitution 
$$\{X(\bar{0}) \ass t_1, X(s(n)) \ass t_2, Y(n,m) \ass t_3\}$$  
is a standard $s$-substitution.
\end{examp}
Note that the parameter unifiability in Definition~\ref{def.ssub-standard} is easy to check: either the variable classes are different, of if they are equal and the expressions are of the form $X(r_1,\ldots,r_k),$ $X(s_1,\ldots,s_k)$ it suffices to find a pair $\{r_i,s_i\}$ which is not parameter unifiable. As $r_i$ and $s_i$ are of the form $n_i, p(n_i), s(n_i), \bar{0}$ this is trivial to check.\\[1ex]
Below we analyze the application of standard $s$-substitutions to standard terms. 
\begin{definition}[state]\label{def.state}
Let $X(r_1,\ldots,r_\alpha)$ be a standard variable expression. By definition we have $r_i \in \{n_i,p(n_i),s(n_i),\bar{0}\}$ for $i=1,\ldots,\alpha$. For every $i$ we define three conditions $C_{i,1},C_{i,2},C_{i,3}$ as follows:
$$C_{i,1} = n_i = \bar{0}, \quad C_{i,2} = n_i \neq \bar{0} \land p(n_i) = \bar{0}, \quad C_{i,3} = n_i \neq \bar{0} \land p(n_i) \neq \bar{0}.$$
Note that $\Ccal_i\colon \{C_{i,1},C_{i,2},C_{i,3}\}$ is a partition. \\[1ex]
A {\em state} for $X(r_1,\ldots,r_\alpha)$ is a conjunction of the form
$$D_1 \land \cdots \land D_\alpha$$
where for all $i =1,\ldots,\alpha$ $D_i \in \Ccal_i$. Note that there exist $3^\alpha$ different states (not all of them are satisfiable) and that the set of all states - written as $\Dcal(X(r_1,\ldots,r_\alpha))$ also forms a partition. Let $p$ be a state for 
$X(r_1,\ldots,r_\alpha)$, $\beta \leq \alpha$ such that $Y(r'_1\ldots,r'_\beta)$ is a standard variable expression; then $p$ is also a state for  $Y(r'_1\ldots,r'_\beta)$. 
\end{definition}
Application of standard $s$-substitutions to standard terms is not a trivial operation. E.g. just by writing $X(r,s)\Theta$ we cannot guarantee the uniqueness of application. 
\begin{examp}\label{ex.appl-non-unique}
Let $\Theta = \{X(n_1,n_2) \ass t\}$ for some term $t$ different from $X(n_1,n_2)$. Consider now the application 
$X(\bar{0},\bar{0})\Theta$.
If $\sigma(n)=\bar{0}$ and $\sigma(m) = \bar{0}$ then $X(\bar{0},\bar{0})\Theta[\sigma] = \sigma(t)\Eval$, for all other parameter assignments we get  $X_i(\bar{0},\bar{0})\Theta[\sigma]=  X_i(\bar{0},\bar{0})$. That means for the state $n=\bar{0} \land m = \bar{0}$ and we get $t$, and for all other states we obtain $X_i(\bar{0},\bar{0})$. We see that a unique result can only be obtained by taking into account the states.
\end{examp}
We are now going to define the application of a standard $s$-substitution w.r.t. the set of states of a standard variable expression. To this aim we define, for any state $p$, a replacement operator $\psi_p$ which performs a partial evaluation of terms and substitutions when the value of some parameters is $\bar{0}$. We define $\psi_p$ for terms in $\Tiotahat$. 
\begin{definition}\label{def.psi-p}
Let $X(r_1,\ldots,r_\alpha)$ be a standard variable expression over the parameters $\{n_1,\ldots,n_\alpha\}$ and let 
$$p\colon D_1 \land \cdots \land D_\alpha$$ 
be a state for $V\colon X(r_1,\ldots,r_\alpha)$. We define $\psi_p$ for $X(r_1,\ldots,r_\alpha)$. Let $p(i)$ be the $i$-th conjunction of $p$, i.e. $p(i) = D_i$. For every $p(i)$ we have to distinguish whether $r_i$ evaluates to $\bar{0}$ under a $\sigma \in \Scal$ such that $\sigma(p(i))  = \top$. For every $p$ we define $\psi_p(V)$ inductively. We define $\psi^0_p(V) = V$. Assume now that for all $j<i$ $\psi^j_p(V)$ are already defined.  We distinguish the following cases below: 
\begin{itemize}
\item If $p(i) = C_{i,1}$ (i.e. we have $p(i)\colon n_i = \bar{0}$) we define $\psi^i_p(V) = \psi^{i-1}_p(V)[r_i/\bar{0}]$ (i.e.  we replace $r_i$ by $\bar{0}$) if $r_i \in \{n_i,p(n_i),\bar{0}\}$ and $\psi^i_p(V) = \psi^{i-1}_p(V)[r_i/\bar{1}]$ if $r_i = s(n_i)$. 
\item If $p(i) = C_{i,2}$ (we have $p(i)\colon n_i \neq \bar{0} \land p(n_i) = \bar{0}$) we define $\psi^i_p(V) = \psi^{i-1}_p(V)[r_i/\bar{0}]$ if $r_i = p(n_i)$, $\psi^i(V) = \psi^{i-1}_p(V)[r_i/\bar{1}]$ if $r_i = n_i$, and 
$\psi^i_p(V) = \psi^{i-1}_p(V)[r_i/\bar{2}]$ if $r_i=s(n_i)$. For $r_i = \bar{0}$ we get $\psi^i_p(V) = \psi^{i-1}_p(V)$. 
\item If $p(i) = C_{i,3}$ (i.e. $p(i) = n_i \neq \bar{0} \land p(n_i) \neq \bar{0}$) then we define $\psi^i_p(V) = \psi^{i-1}_p(V)$.
\end{itemize}
Finally we define $\psi_p(V) = \psi^\alpha_p(V)$. \\[1ex]
We extend now $\psi_p$ to $\Tiota$:
\begin{itemize}
 \item If $t \in \Fiota_0$ then $\psi_p(t) = t$.
\item  For a standard variable expressions $X(r_1,\ldots,r_\alpha)$ $\psi_p(X(r_1,\ldots,r_\beta))$ is defined above.
\item If $t = f(t_1,\ldots,t_k)$ then $\psi_p(f(t_1,\ldots,t_k)) = f(\psi_p(t_1),\ldots,\psi_p(t_k))$.
\end{itemize}
Finally we extend $\psi_p$ to $\Tiotahat$:
\begin{itemize}
\item If $t \in \Tiota$ then $\psi_p(t)$ is defined as above.
\item If $t= f(t_1,\ldots,t_k)$  for $f \in \Fiota_k$ then 
$$\psi_p(f(t_1,\ldots,t_k)) =  f(\psi_p(t_1),\ldots,\psi_p(t_k)).$$
\item if $t =  \that(s_1,\ldots,s_j,t_1,\ldots,t_k)$ for a profile $\pi = (j,k)$ and $\that \in \Fhat(\pi)$, $s_l \in \Tiotahat$ for $l=1,\ldots,j$, $t_1,\ldots,t_k \in \Tomega$ then 
$$\psi_p(t)   =\that(\psi_p(s_1),\ldots,\psi_p(s_j),t_1,\ldots,t_k).$$
\end{itemize}
\end{definition}
\begin{examp}\label{ex.psi-p}
Let $t = h(X(n_1,n_2),Y(\bar{0},s(n_2)))$ for $h \in \Fiota_2$ and 
$p = n_1 \neq \bar{0} \land p(n_1) \neq \bar{0} \land n_2 = \bar{0}$. Then 
$$\psi_p(t) = h(X(n_1,\bar{0}),Y(\bar{0},\bar{1})).$$
\end{examp}
The terms $t$ and $\psi_p(t)$ evaluate to the same first-order term if $\sigma(p)\Eval = \top$, a property we will need in soundness results in this section.
\begin{lemma}\label{le.psi-p}
Let $t$ be a standard term in $\Tiotahat$ such that the set of parameters in $t$ is contained in $\Pcal\colon \{n_1,\ldots,n_\alpha\}$,  let $\Scal(p)$ be the set of all parameter assignments such that $\sigma(p)\Eval = \top$. Then 
$\sigma(t)\Eval =  \sigma(\psi_p(t))\Eval$ for  all states $p$ over $\Pcal$ and $\sigma \in \Scal(p)$.
\end{lemma}
\begin{proof}
By the homomorphic extension of $\psi_p$ from variable expressions to terms, it is enough to show that for all variable expressions 
$X(r_1,\ldots,r_\beta)$ occurring in $t$ (where $\beta \leq \alpha$) we obtain  
$$(\star)\ \sigma(X(r_1,\ldots,r_\beta))\Eval =  \sigma(\psi_p(X(r_1,\ldots,r_\beta)))\Eval \mbox{ for } \sigma \in \Scal(p).$$ 
So let us assume we have the variable expression $X(r_1,\ldots,r_\beta)$ in $t$. In order to prove $(\star)$ we recall the definition of 
$\psi_p$ via $\psi^i_p$ as in Definition~\ref{def.psi-p} and proceed by induction on $i$. \\[1ex]
$i = 0$: by definition we have $\psi^0_p(X(r_1,\ldots,r_\beta)) = X(r_1,\ldots,r_\beta)$ and $(\star)$ trivially holds when we replace 
$\psi_p$ by $\psi^0_p$. Now let us assume that 
$${\rm (IH)}\ \sigma(X(r_1,\ldots,r_\beta))\Eval =  \sigma(\psi^j_p(X(r_1,\ldots,r_\beta)))\Eval \mbox{ for } j<i.$$
By definition of the $\psi^j_p$ there exist terms $r'_1,\ldots,r'_{i-1}$ such that 
$$\psi^{i-1}_p(X(r_1,\ldots,r_\beta)) = X(r'_1,\ldots,r'_{i-1},r_i,\ldots,r_\beta).$$
In order to define $\psi^i_p$ from $\psi^{i-1}_p$ we have to consider the condition $p(i)$. We distinguish three cases 
\begin{itemize}
\item $p(i) = n_i = \bar{0}$: as $\sigma(p)\Eval = \top$ we also have $\sigma(p(i))\Eval = \top$ and therefore $\sigma(n_i) = \bar{0}$. By definition of $\psi^i_p$ we have 
$$\psi^i_p(X(r'_1,\ldots,r'_{i-1},r_i,\ldots,r_\beta) = X(r'_1,\ldots,r'_{i-1},\bar{0},\ldots,r_\beta)$$ 
if $r_i \in \{n_i,p(n_i),\bar{0}\}$; in this case we obtain $\sigma(r_i)\Eval = \sigma(\bar{0})\Eval = \bar{0}$ and therefore 
 $$\sigma(X(r_1,\ldots,r_\beta))\Eval =  \sigma(\psi^i_p(X(r_1,\ldots,r_\beta)))\Eval .$$
If $r_i = s(n_i)$ we obtain 
$$\psi^i_p(X(r'_1,\ldots,r'_{i-1},r_i,\ldots,r_\beta) = X(r'_1,\ldots,r'_{i-1},\bar{1},\ldots,r_\beta).$$ 
But then $\sigma(r_i)\Eval = \bar{1} = \sigma(\bar{1})\Eval$, and again 
 $$\sigma(X(r_1,\ldots,r_\beta))\Eval =  \sigma(\psi^i_p(X(r_1,\ldots,r_\beta)))\Eval .$$
\item $p(i) = n_i \neq \bar{0} \land p(n_i) = \bar{0}$: as $\sigma(p)\Eval = \top$ we get $\sigma(n_i)=\bar{1}, \sigma(p(n_i))= \bar{0}$. By definition of $\chi$ we obtain 
\begin{eqnarray*}
\psi^i_p(X(r'_1,\ldots,r'_{i-1},r_i,\ldots,r_\beta)) &=& X(r'_1,\ldots,r'_{i-1},\bar{0},\ldots,r_\beta) \mbox{ if } r_i = p(n_i),\\
                        &=& X(r'_1,\ldots,r'_{i-1},\bar{1},\ldots,r_\beta) \mbox{ if } r_i = n_i,\\
                        &=& X(r'_1,\ldots,r'_{i-1},\bar{2},\ldots,r_\beta) \mbox{ if } r_i = s(n_i),\\                      
                        &=& \psi^{i-1}p (X(r'_1,\ldots,r'_{i-1},r_i,\ldots,r_\beta))\mbox{ if } r_i = \bar{0}.
\end{eqnarray*}
It is easy to see, that by $\sigma(C_{i,2})=\top$ we obtain 
 $$\sigma(X(r_1,\ldots,r_\beta))\Eval =  \sigma(\psi^i_p(X(r_1,\ldots,r_\beta)))\Eval .$$
\item $p(i) = n_i \neq \bar{0} \land p(n_i) \neq \bar{0}$: by definition of $\psi_p$ we get $\psi^i_p = \psi^{i-1}_p$ and thus by (IH) 
$$\sigma(X(r_1,\ldots,r_\beta))\Eval =  \sigma(\psi^i_p(X(r_1,\ldots,r_\beta)))\Eval .$$
\end{itemize}
\end{proof} 
We extend $\psi_p$ further to standard $s$-substitutions:
\begin{definition}\label{def.psi-to-subst}
Let $\Theta\colon \{A_1 \ass t_1,\ldots, A_\alpha \ass t_\alpha\}$ be a standard $s$-substitution over the parameters $\{n_1,\ldots,n_\beta\}$ and $p$ be a state over these parameters. Then
$$\psi_p(\Theta) = \{\psi_p(A_1) \ass \psi_p(t_1),\ldots, \psi_p(A_\alpha) \ass \psi_p(t_\alpha)\}.$$
\end{definition}
\begin{examp}\label{ex.psi-to-subst}
Let $\Theta = \{X(\bar{0}) \ass Z(n_1),\  X(s(n_1)) \ass g(Z(n_1)),\  Y_1(n_1,n_2) \ass h(Y_2(n_1,n_2))\}$ and $p = n_1 \neq \bar{0} \land p(n_1) = \bar{0} \land n_2 \neq \bar{0} \land p(n_2) \neq\bar{0}.$
Then 
$$\psi_p(\Theta) =  \{X(\bar{0}) \ass Z(\bar{1}),\  X(\bar{2}) \ass g(Z(\bar{1})),\  Y_1(\bar{1},n_2) \ass h(Y_2(\bar{1},n_2))\}.
$$
\end{examp}
\begin{proposition}\label{psi-subst}
Let $\Theta$ be a standard $s$-substitution, $p$ be a state over the parameters of $\Theta$. Then $\psi_p(\Theta)$ is an extended $s$-substitution and for all parameter assignments $\sigma \in \Scal(p)$ we have $\Theta[\sigma] = \psi_p(\Theta)[\sigma]$.
\end{proposition}
\begin{proof}
Let $\Theta = \{A_1 \ass t_1,\ldots, A_\alpha \ass t_\alpha\}$. Then 
$$\psi_p(\Theta) = \{\psi_p(A_1) \ass \psi_p(t_1), \ldots, \psi_p(A_\alpha) \ass \psi_p(t_\alpha)\}.$$
As $\Theta$ is a standard $s$-substitution we have for all $A_i,A_j$ such that $i \neq j$ and all $\sigma \in \Scal$ that $\sigma(A_i)\Eval \neq \sigma(A_j)\Eval$. Now consider $A_i,A_j$ for  $i \neq j$ and $\psi_p(A_i),\psi_p(A_j)$. Now assume that $\sigma(\psi_p(A_i))\Eval = \sigma(\psi_p(A_j))\Eval$ for some $\sigma \in \Scal$. Let $A_i = X(r_1,\ldots,r_\alpha)$ and $A_j = X(s_1,\ldots,s_\alpha)$. 
Then $\psi_p(A_i) = X(r_1',\ldots,r'_\alpha)$ and $\psi_p(A_j) = X(s'_1,\ldots,s'_\alpha)$ where, for all $i=1,\ldots,\alpha$ either $s_i = s'_i, r_i = r'_i$ or $r'_i \in \{\bar{0},\bar{1},\bar{2}\}$ and $s'_i \in \{\bar{0},\bar{1},\bar{2}\}$. As $\sigma(\psi_p(A_i))\Eval = \sigma(\psi_p(A_j))\Eval$ we have if, for some $i$, $r'_i \in \{\bar{0},\bar{1},\bar{2}\}$ and $s'_i \in \{\bar{0},\bar{1},\bar{2}\}$ ($s'_i,r'_i$ are thus numerals) then $r'_i = s'_i$. Now let $J \IN I$ (for $I = \{1,\ldots,\alpha\}$) be the subset of positions in $\{1,\ldots,\alpha\}$ where $r'_i = s'_i$ and $r'_i,s'_i \in \{\bar{0},\bar{1},\bar{2}\}$
As $\sigma(\psi_p(A_i))\Eval = \sigma(\psi_p(A_j))\Eval$ we must have $\sigma(r_j)\Eval = \sigma(s_j)\Eval$ for all $j \in I \setminus J$. 
We define now $\sigma'(n_j) = \sigma(n_j)$ for $j \in I \setminus J$ and $\sigma'(n_j)$ for $j \in J$ is defined in a way that 
$\sigma'(r_j)\Eval = \bar{k}$ for $r'_j = \bar{k}$. 
Then 
$$\sigma'(A_i)\Eval = \sigma'(\psi_p(A_i))\Eval = \sigma'(\psi_p(A_j))\Eval = \sigma'(A_j)\Eval$$
contradicting the assumption that $\Theta$ is an $s$-substitution.\\[1ex]
That for $\sigma \in \Scal(p)$ we have $\Theta[\sigma] = \psi_p(\Theta)[\sigma]$ is trivial by 
$$\sigma(A_i)\Eval = \sigma(\psi_p(A_i))\Eval \mbox{ and } \sigma(t_i)\Eval = \sigma(\psi_p(t_i))\Eval \mbox{ for }i=1,\ldots,\alpha$$ 
which follows from Lemma~\ref{le.psi-p}. 
\end{proof}
We are now in the possession of the tools to define the application of a standard $s$-substitution to a standard term.
\begin{definition}\label{def.appl-standard}
Let $\Theta$ be a standard $s$-substitution and $t$ be a standard term and $\Dcal$ the set of states over the parameters occurring in $\Theta$ and $s$. Then we define 
$$\Ap(\Theta,t) = \{p\colon \psi_p(t)\psi_p(\Theta) \mid p \in \Dcal\}.$$
\end{definition}
Note that we can compactify the representation of $\Ap(\Theta,t)$ in Definition~\ref{def.appl-standard}. If there are just a few states, let us call them $\{p_1,\ldots,p_\beta\}$, such that $\psi_p(t)\psi_p(\Theta) \neq t$ for $p \in \{p_1,\ldots,p_\beta\}$ and 
$\psi_{p}(t)\psi_p(\Theta) = t$ for $p \in \Dcal' = \Dcal \setminus \{p_1,\ldots,p_\beta\}$ we can write  
$$\Ap(\Theta,t) = \{p_1\colon \psi_{p_1}(t)\psi_{p_1}(\Theta),\ldots, p_\beta\colon \psi_{p_\beta}(t)\psi_{p_\beta}(\Theta)\} \union \{\Dcal'\colon t \}.$$
\begin{definition}[semantics of $\Ap$]\label{def.sem-appl-standard}
Let $\Ap(\Theta,t)$ be defined as in Definition~\ref{def.appl-standard}. Then 
\begin{eqnarray*}
\Ap(\Theta,t)|_p &=& \psi_p(t)\psi_p(\Theta),\\
\Ap(\Theta,t)[\sigma] &=& \sigma(\Ap(\Theta,t)|_p)\Eval \mbox{ if } \sigma(p)\Eval = \top.
\end{eqnarray*}
\end{definition}
Note that, as the set of all states form a partition, $\Ap(\Theta,t)[\sigma]$ in Definition~\ref{def.sem-appl-standard} is defined for all $\sigma \in \Scal$.
\begin{examp}\label{ex.appl-standard}
Let $\Theta = \{X(n_1,n_2) \ass g(Y(n_1))\}$ like in Example~\ref{ex.appl-non-unique}. Consider now the syntactic application 
$X(\bar{0},\bar{0})\Theta$ which yields $X(\bar{0},\bar{0})$. Obviously $X(\bar{0},\bar{0})\Theta$ is unsound for $\sigma(n_1) = \sigma(n_2) = \bar{0}$ as then 
$$\sigma(X(\bar{0},\bar{0})\Theta)\Eval  = X(\bar{0},\bar{0}) \neq g(Y(\bar{0})) = X(\bar{0},\bar{0})\Theta[\sigma].$$
The problem is resolved with the states $p \in \Dcal(X(n_1,n_2))$.
Let $p_1 = n_1 = \bar{0} \land n_2 = \bar{0}$, then 
$$\psi_p(X(\bar{0},\bar{0}))\psi_p(\Theta) = X(\bar{0},\bar{0})\{X(\bar{0},\bar{0}) \ass g(Y(\bar{0})) = g(Y(\bar{0})) .$$
For the other $8$ states $p' \in \{p_2,\ldots,p_9\}$ we get 
\[
\begin{array}{l}
\psi_{p'}(X(\bar{0},\bar{0}))\psi_{p'} (\Theta) = X(\bar{0},\bar{0}), \mbox{ and therefore}\\
\Ap(\Theta,t) = \{p_1\colon g(Y(\bar{0}))\} \union \{\{p_2,\ldots,p_9\}\colon X(\bar{0},\bar{0})\}. 
\end{array}
\]
\end{examp}
We show now that the definition of application is sound. To this aim we need the following lemma which shows that {\em syntactic} application of substitutions is sound for terms $\psi_p(t)$ and substitutions $\psi_p(\Theta)$.
\begin{lemma}\label{le.psi-p2}
Let $\Theta$ be a standard $s$-substitution and $t$ be a standard term. Then for all states $p$ over the parameters of $\Theta$ and $t$ and for all $\sigma \in \Scal(p)$ we have $\sigma(\psi_p(t))\Eval \psi_p(\Theta)[\sigma] = \sigma(\psi_p(t)\psi_p(\Theta))\Eval$. 
\end{lemma}
\begin{proof}
It is enough to show that $\sigma(\psi_p(B))\Eval\psi_p(\Theta)[\sigma] =  \sigma(\psi_p(B)\psi_p(\Theta))\Eval$ for variable expressions $B$ and $\sigma \in \Scal(p)$. So let $B = X(r_1,\ldots,r_\beta)$ and 
$\Theta = \{A_1 \ass t_1,\ldots, A_\gamma \ass t_\gamma\}$,
where $B$ and $\Theta$ range over the parameters $\{n_1,\ldots,n_\alpha\}$. If $X$ does not occur in the domain of $\Theta$ then for all $\sigma$
$$\sigma(\psi_p(B))\Eval\psi_p(\Theta)[\sigma] = \sigma(\psi_p(B))\Eval = \sigma(\psi_p(B)\psi_p(\Theta))\Eval$$
and the lemma trivially holds. Therefore it remains to assume $X(s_1,\ldots,s_\beta) \ass t \in \Theta$ where $t$ is some of the $t_1,\ldots,t_\gamma$. Let $r'_i = \psi_p(r_i)$ and $s'_i = \psi_p(s_i)$ for $i=1,\ldots,\beta$. We show that
\begin{itemize}
\item $X(r'_1,\ldots,r'_\beta) = X(s'_1,\ldots,s'_\beta)$ iff there exists a $\sigma \in \Scal(p)$ such that 
$$(\star)\sigma(X(r'_1,\ldots,r'_\beta))\Eval = \sigma(X(s'_1,\ldots,s'_\beta))\Eval.$$
\end{itemize}
The direction from left to right is trivial. So it is enough to prove that the existence of a $\sigma$ which yields $(\star)$ implies that 
$X(r'_1,\ldots,r'_\beta) = X(s'_1,\ldots,s'_\beta)$. We assume $(\star)$ and $X(r'_1,\ldots,r'_\beta) \neq X(s'_1,\ldots,s'_\beta)$ and derive a contradiction. So let us assume that $r'_i \neq s'_i$ and consider the definition of $\psi^i_p(r'_i),\psi^i(s'_i)$ in Definition~\ref{def.psi-p}: 
\begin{itemize}
\item If $p(i) = C_{i,1}$ (i.e. we have $p(i)\colon n_i = \bar{0}$) we define $\psi^i_p(V) = \psi^{i-1}_p(V)[r_i/\bar{0}]$ (i.e.  we replace $r_i$ by $\bar{0}$) if $r_i \in \{n_i,p(n_i),\bar{0}\}$ and $\psi^i_p(V) = \psi^{i-1}_p(V)[r_i/\bar{1}]$ if $r_i = s(n_i)$. That means that, in this case $r'_i$ and $s'_i$ are constant symbols and, as $\sigma(X(r'_1,\ldots,r'_\beta))\Eval = \sigma(X(s'_1,\ldots,s'_\beta))\Eval$ and thus also $\sigma(r'_i)\Eval = \sigma(s'_i)\Eval)$ - which implies that $r'_i = s'_i$.
\item If $p(i) = C_{i,2}$ (we have $p(i)\colon n_i \neq \bar{0} \land p(n_i) = \bar{0}$) we define $\psi^i_p(V) = \psi^{i-1}_p(V)[r_i/\bar{0}]$ if $r_i = p(n_i)$, $\psi^i(V) = \psi^{i-1}_p(V)[r_i/\bar{1}]$ if $r_i = n_i$, and $\psi^i_p(V) = \psi^{i-1}_p(V)[r_i/\bar{2}]$ if $r_i=s(n_i)$. For $r_i = \bar{0}$ we get $\psi^i_p(V) = \psi^{i-1}_p(V)$. Here $r'_i$ and $s'_i$ are both constants and therefore, as in the case above, $r'_i = s'_i$ as, otherwise, $\sigma(r'_i)\Eval \neq \sigma(s'_i)\Eval$.
\item If $p(i) = C_{i,3}$ (i.e. $p(i) = n_i \neq \bar{0} \land p(n_i) \neq \bar{0}$) then we define $\psi^i_p(V) = \psi^{i-1}_p(V)$. We consider all cases where $r'_i \neq s'_i$:
\begin{enumerate}
\item $r'_i = \bar{0}, s'_i = n_i$: as $\sigma(n_i) \neq \bar{0}$ we get $\sigma(r'_i)\Eval \neq \sigma(s'_i)\Eval$.
\item  $r'_i = \bar{0}, s'_i = p(n_i)$:  as $\sigma(p(n_i)) \neq \bar{0}$ we get $\sigma(r'_i)\Eval \neq \sigma(s'_i)\Eval$.
\item $r'_i = \bar{0}, s'_i = s(n_i)$: $\sigma(r'_i)\Eval \neq \sigma(s'_i)\Eval$.
\item $r'_i = n_i, s'_i = s(n_i)$: for all $\sigma$  $\sigma(r'_i)\Eval \neq \sigma(s'_i)\Eval$.
\item $r'_i = n_i, s'_i = p(n_i)$: for $\sigma \in \Scal(p)$ we have $\sigma(p(n_i))\Eval \neq \bar{0}$ and thus 
$\sigma(p(n_i))\Eval \neq \sigma(n_i)$ which yields  $\sigma(r'_i)\Eval \neq \sigma(s'_i)\Eval$.
\item $r'_i = p(n_i), s'_i = s(n_i)$:  for all $\sigma$  $\sigma(r'_i)\Eval \neq \sigma(s'_i)\Eval$.
\end{enumerate}
\end{itemize}
As all remaining cases are symmetric (it does not matter whether we start with $r'_i$ or $s'_i$) we obtain that for all 
$\sigma \in \Scal(p)$ $\sigma(r'_i)\Eval \neq \sigma(s'_i)\Eval$ - which contradicts $(\star)$. 
We can conclude that $\{X(r'_1,\ldots,r'_\beta) ,X(s'_1,\ldots,s'_\beta)\}$ is parameter-unifiable iff 
$X(r'_1,\ldots,r'_\beta) =X(s'_1,\ldots,s'_\beta)$. Therefore 
\[
\begin{array}{l}
X(r'_1,\ldots,r'_\beta)\{ X(s'_1,\ldots,s'_\beta) \ass \psi(t)\} = \psi(t) \mbox{ iff  for all }\sigma \in \Scal(p)\\
\sigma(X(r'_1,\ldots,r'_\beta))\Eval\{ \sigma(X(s'_1,\ldots,s'_\beta))\Eval \ass \sigma(\psi(t))\Eval\} = \sigma(\psi(t)) \Eval
\end{array}
\]
which in the end yields $\sigma(\psi_p(t))\Eval \psi_p(\Theta)[\sigma] = \sigma(\psi_p(t)\psi_p(\Theta))\Eval$ for $\sigma \in \Scal(p)$.
\end{proof}
\begin{theorem}\label{the.soundness-appl}
Application is sound, i.e. for all $\sigma \in \Scal(p)$ we have $\sigma(\Ap(\Theta,t)|_p)\Eval = \sigma(t)\Eval\Theta[\sigma]$.
\end{theorem}
\begin{proof}
Let $\Ap(\Theta,t)$ be defined as in Definition~\ref{def.appl-standard}. Then 
\begin{eqnarray*}
\Ap(\Theta,t)|_p &=& \psi_p(t)\psi_p(\Theta) \mbox{ and therefore }\\
\sigma(\Ap(\Theta,t)|_p)\Eval &=& \sigma(\psi_p(t)\psi_p(\Theta))\Eval \mbox{ and by Lemma}~\ref{le.psi-p2}\\
 \sigma(\psi_p(t))\Eval \psi_p(\Theta)[\sigma] &=& \sigma(\psi_p(t)\psi_p(\Theta))\Eval.
\end{eqnarray*}
As $\sigma(\psi_p(t))\Eval = \sigma(t)\Eval$ and $\Theta[\sigma] = \psi_p(\Theta)[\sigma]$ for $\sigma \in \Scal(p)$ the theorem follows.
\end{proof}
We illustrate now the application of a substitution $\Theta$ in case of variable expressions $X(r_1,r_2)$ for $X \in \Viota_2$.
\begin{examp}\label{ex.appl-Viota-2}
Given a standard variable expression $X(r_1,r_2)$ and $\Theta = \Theta_1 \union \cdots \Theta_i$ we compute the term $X(r_1,r_2)\Theta$. We assume that the  $\Theta_i$ are the substitutions for the variable class $X_i$. Therefore it is sufficient to determine the result of $X(s,t)\Theta_i$ for just one $i \in \{1,\ldots,\alpha\}$: if $X \neq X_i$ then we have $X(s,t)\Theta_i = \{\top\colon X(s,t)\}$.

To this aim we have to distinguish all different forms of $X(r_1,r_2)$ and all forms of replacements $X(r'_1,r'_2) \ass t'$ in $\Theta_i$ and all states $p$ of $X(r_1,r_2)$ where $X$ is some of the $X_i$. According to Definition~\ref{def.state} we obtain $9$ states i.e. $\Dcal = \{p_1,\ldots,p_9\}$ (for comfort we use the names $n,m$ for $n_1,n_2$). 
\begin{eqnarray*}
p_1 &=& n = \bar{0} \land m = \bar{0},\\
p_2 &=& n = \bar{0} \land m \neq \bar{0} \land p(m) \neq \bar{0},\\
p_3 &=& n = \bar{0} \land m \neq \bar{0} \land p(m) = \bar{0},\\
p_4 &=& n \neq \bar{0} \land m = \bar{0} \land p(n) \neq \bar{0},\\
p_5 &=&  n \neq \bar{0} \land m = \bar{0} \land p(n) = \bar{0},\\
p_6 &=&  n \neq \bar{0} \land m \neq \bar{0} \land p(n) = \bar{0} \land p(m) = \bar{0},\\
p_7 &=& n \neq \bar{0} \land m \neq \bar{0} \land p(n) = \bar{0} \land p(m) \neq \bar{0},\\
p_8 &=& n \neq \bar{0} \land m \neq \bar{0} \land p(n) \neq \bar{0} \land p(m) = \bar{0},\\
p_9 &=& n \neq \bar{0} \land m \neq \bar{0} \land p(n) \neq \bar{0} \land p(m) \neq \bar{0}.
\end{eqnarray*}
Let us assume that $X(r_1,r_2) = X(n,m)$ and $\Theta_i = \{X(\bar{0},m) \ass g(Y(n)), X(s(n),m) \ass Y(n)\}$.  
According to Definition~\ref{def.appl-standard} we have to determine the set 
$$S =\{p_1\colon \psi_1(X(n,m))\psi_1(\Theta_i), \ldots, p_9\colon \psi_9(X(n,m))\psi_9(\Theta_i)\}$$
where we write $\psi_i$ for $\psi_{p_i}$. 
\[
\begin{array}{l}
\psi_1(X(n,m)) = X(\bar{0},\bar{0}),\ \psi_1(\Theta_i)=  \{X(\bar{0},\bar{0}) \ass g(Y(\bar{0})), 
X(s(\bar{0}),\bar{0}) \ass Y(\bar{0})\}\\
\mbox{therefore } p_1\colon X(\bar{0},\bar{0})\{X(\bar{0},\bar{0}) \ass g(Y(\bar{0})), 
X(s(\bar{0}),\bar{0}) \ass Y(\bar{0})\} \in S.\\
\mbox{resulting in }p_1\colon g(Y(\bar{0})) \in S.\\[1ex]
\psi_2(X(n,m)) = X(\bar{0},m),\ \psi_2(\Theta_i)=  \{X(\bar{0},m) \ass g(Y(\bar{0})), 
X(s(\bar{0}),m) \ass Y(\bar{0})\}\\
\mbox{therefore }p_2\colon  g(Y(\bar{0})) \in S.\\[1ex]
\psi_3(X(n,m)) = X(\bar{0},\bar{1}),\ \psi_3(\Theta_i)=  \{X(\bar{0},\bar{1}) \ass g(Y(\bar{0})), 
X(\bar{1},\bar{1}) \ass Y(\bar{0})\}\\
\mbox{therefore }p_3\colon  g(Y(\bar{0})) \in S.\\[1ex]
\psi_4(X(n,m)) = X(n,\bar{0}),\ \psi_4(\Theta_i)=  \{X(\bar{0},\bar{0}) \ass g(Y(n)), 
X(s(n),\bar{0}) \ass Y(n)\}\\
\mbox{therefore }p_4\colon  X(n,\bar{0}) \in S.\\[1ex]
\psi_5(X(n,m)) = X(\bar{1},\bar{0}),\ \psi_5(\Theta_i)=  \{X(\bar{0},\bar{0}) \ass g(Y(\bar{1})), 
X(\bar{2},\bar{0}) \ass Y(\bar{1})\}\\
\mbox{therefore }p_5\colon  X(\bar{1},\bar{0}) \in S.\\[1ex]
\psi_6(X(n,m)) = X(\bar{1},\bar{1}),\ \psi_6(\Theta_i)=  \{X(\bar{0},\bar{1}) \ass g(Y(\bar{1})), 
X(\bar{2},\bar{1}) \ass Y(\bar{1})\}\\
\mbox{therefore }p_6\colon  X(\bar{1},\bar{1}) \in S.\\[1ex]
\psi_7(X(n,m)) = X(\bar{1},m),\ \psi_7(\Theta_i)=  \{X(\bar{0},m) \ass g(Y(\bar{1})), 
X(\bar{2},m) \ass Y(\bar{1})\}\\
\mbox{therefore }p_7\colon  X(\bar{1},m) \in S.\\[1ex]
\psi_8(X(n,m)) = X(n,\bar{1}),\ \psi_8(\Theta_i)=  \{X(\bar{0},\bar{1}) \ass g(Y(\bar{n})), 
X(s(n),\bar{1}) \ass Y(\bar{1})\}\\
\mbox{therefore }p_8\colon  X(n,\bar{1}) \in S.\\[1ex]
\psi_9(X(n,m) = X(n,m),\ \psi_9(\Theta_i) = \Theta_i \mbox{ and therefore } p_9\colon X(n,m) \in S.
\end{array}
\]
Now we can determine $\Ap(\Theta_i,X(n,m))$; we obtain 
\begin{eqnarray*}
\Ap(\Theta_i,X(n,m)) &=& \{\{p_1,p_2,p_3\}\colon  g(Y(\bar{0})),\ p_4\colon  X(n,\bar{0}),\  p_5\colon  X(\bar{1},\bar{0}) , \\
&& p_6\colon  X(\bar{1},\bar{1}),\ p_7\colon  X(\bar{1},m), \  p_8\colon  X(n,\bar{1}),\ p_9\colon X(n,m)\}.
\end{eqnarray*}
It is easy to see that, e.g., $\Ap(\Theta_i,X(s(n),s(m))) = \{\Dcal: X(s(n),s(m))\}$. 
\end{examp}
\begin{examp}\label{ex.standard-appl-to-term}
Let $\Theta = \Theta_i$ be the standard $s$-substitution defined in Example~\ref{ex.appl-Viota-2}. We consider the term 
$$t = r(X(n,m),X(s(n),\bar{0}),Y(n)).$$ 
By Definition~\ref{def.ssub-Tiota} we have, for every state $p$, 
$$\Ap(\Theta,t)_p = r(\psi_p(X(n,m))\chi_p(\Theta),\psi_p(X(s(n),\bar{0})\chi_p(\Theta),\psi_p(Y(n))\chi_p(\Theta).$$
Again we have only two parameters so $\Dcal(X(n,m)) = \{p_1,\ldots,p_9\}$ as in Example~\ref{ex.appl-Viota-2}. Therefore we obtain 
$$\Ap(\Theta,t) = \{p\colon \Ap(\Theta,t)|_p \mid p \in \{p_1,\ldots,p_9\}\}$$
Below we compute the subset $S\colon \{\Ap(\Theta,t)|_{p_1}, \Ap(\Theta,t)|_{p_2}, \Ap(\Theta,t)|_{p_9}\}$ of 
$\Ap(\Theta,t)$.
Recall that 
$$\Theta = \{X(\bar{0},m) \ass g(Y(n)), X(s(n),m) \ass Y(n)\} \mbox{ and }$$
\begin{eqnarray*}
p_1 &=& n = \bar{0} \land m = \bar{0},\\
p_2 &=& n = \bar{0} \land m \neq \bar{0} \land p(m) \neq \bar{0},\\
p_9 &=& n \neq \bar{0} \land m \neq \bar{0} \land p(n) \neq \bar{0} \land p(m) \neq \bar{0}.
\end{eqnarray*}
\begin{eqnarray*}
\Ap(\Theta,t)|_{p_1} &=& r(g(Y(\bar{0})),Y(\bar{0}),Y(n))),\\
\Ap(\Theta,t)|_{p_2} &=& r(g(Y(\bar{0})),Y(\bar{0}),Y(n))),\\ 
\Ap(\Theta,t)|_{p_9} &=& r(X(n,m),X(s(n),\bar{0}),Y(n)).
\end{eqnarray*}
Therefore $S = \{\{p_1,p_2\}\colon r(g(Y(\bar{0}))Y(\bar{0}),Y(n))),\ p_9\colon r(X(n,m),X(s(n),\bar{0}),Y(n))\}$.
\end{examp}
Though the computation of application  is still complicated - also for standard terms and standard $s$-substitutions - the complexity is fixed when the set of parameters if fixed. The number of different states for a standard substitution over two parameters is always 9  - thus a fixed number. Fixed is also the number of possible arguments for a variable $X \in \Viota_2$ and $X(n,m)$, namely $(r,s)\colon r \in \{\bar{0},n,s(n),p(n)\}, s \in  \{\bar{0},m,s(m),p(m)\}$ - i.e. 16. 
In fact the substitutions can be computed via predefined tables! Fixed numbers (though larger ones) we get also for more than 2 parameters. Thus the complexity of the representation only depends on the number of parameters and not on the complexity of terms. This holds also for the concatenation of substitutions (under a given state concatenation of substitutions are uniquely defined).
We are now able to define the composition of standard $s$-substitutions:
\begin{definition}\label{def.comp-standard-ssub}
Let $\Theta_1$ and $\Theta_2 $ be standard $s$-substitutions such  that the set of parameters in the domains and range of $\Theta_1,\Theta_2$ is contained in $\{n_1,\ldots,n_\gamma\}$. Then for any $p \in \Dcal(X(n_1,\ldots,n_\gamma))$ we define 
\begin{eqnarray*}
(\Theta_1 \circ \Theta_2)|_p &=& \psi_p(\Theta_1)\psi_p(\Theta_2).
\end{eqnarray*}
Finally, $\Theta_1 \circ \Theta_2 = \{p\colon  (\Theta_1 \circ \Theta_2)|_p \mid p \in \Dcal(X(n_1,\ldots,n_\gamma))\}$.
\end{definition}
\begin{examp}\label{ex.comp-s-subst}
Let 
\begin{eqnarray*}
\Theta_1 &=& \{X(\bar{0},m) \ass g(Y(n)), X(s(n),m) \ass Y(n)\},\\
\Theta_2 &=& \{Y(n) \ass \fhat(Z(n),n), Y_1(\bar{0},\bar{0},k) \ass a\},\\
\fhat(x,\bar{0}) &=& x,\\
\fhat(x,s(n)) &=& f(\fhat(x,n)).
\end{eqnarray*}
As we have $3$ parameters in $\Theta_1,\Theta_2$ we have $27$ different states. We select $3$ states $p_1,p_2,p_3$ of them and compute $(\Theta_1 \circ \Theta_2)_p$. Let 
\begin{eqnarray*}
p_1 &=& n = \bar{0} \land m = \bar{0} \land k = \bar{0},\\
p_2 &=& n \neq \bar{0} \land p(n) = \bar{0} \land m = \bar{0} \land k \neq \bar{0} \land p(k) \neq \bar{0},\\
p_3 &=& n \neq \bar{0} \land p(n) \neq \bar{0} \land m \neq \bar{0} \land p(m) \neq \bar{0} \land  k \neq \bar{0} \land p(k) \neq \bar{0}.
\end{eqnarray*}
Then 
\[
\begin{array}{l}
(\Theta_1 \circ \Theta_2)|_{p_1} = \psi_{p_1}(\Theta_1)\psi_{p_1}(\Theta_2) =\\
\{X(\bar{0},\bar{0}) \ass g(Y(\bar{0})), X(\bar{1},\bar{0}) \ass Y(\bar{0})\}\{Y(\bar{0}) \ass \fhat(Z(\bar{0}),\bar{0}), Y_1(\bar{0},\bar{0},\bar{0}) \ass a\} =\\
\{X(\bar{0},\bar{0}) \ass g(\fhat(Z(\bar{0}),\bar{0})),\ X(\bar{1},\bar{0}) \ass \fhat(Z(\bar{0}),\bar{0}),\\
Y(\bar{0}) \ass \fhat(Z(\bar{0}),\bar{0}), Y_1(\bar{0},\bar{0},\bar{0}) \ass a\} = \\
\{X(\bar{0},\bar{0}) \ass g(Z(\bar{0}),\bar{0})),\ X(\bar{1},\bar{0}) \ass Z(\bar{0}), Y(\bar{0}) \ass Z(\bar{0}),\ Y_1(\bar{0},\bar{0},\bar{0}) \ass a\}.\\[1ex]
(\Theta_1 \circ \Theta_2)|_{p_2} = \psi_{p_2}(\Theta_1)\psi_{p_2}(\Theta_2) =\\
\{X(\bar{0},\bar{0}) \ass g(Y(\bar{1})),\ X(\bar{2},\bar{0}) \ass Y(\bar{1})\}\{Y(\bar{1}) \ass \fhat(Z(\bar{1}),\bar{1}), Y_1(\bar{0},\bar{0},k) \ass a\} =\\
\{X(\bar{0},\bar{0}) \ass g(\fhat(Z(\bar{1}),\bar{1})),\ X(\bar{2},\bar{0}) \ass \fhat(Z(\bar{1}),\bar{1}),\\ 
Y(\bar{1}) \ass \fhat(Z(\bar{1}),\bar{1}), Y_1(\bar{0},\bar{0},k) \ass a\} = \\
\{X(\bar{0},\bar{0}) \ass g(f(Z(\bar{1}))),\ X(\bar{2},\bar{0}) \ass f(Z(\bar{1})), Y(\bar{1}) \ass f(Z(\bar{1})), Y_1(\bar{0},\bar{0},k) \ass a\}.\\[1ex]
(\Theta_1 \circ \Theta_2)|_{p_3} = \psi_{p_3}(\Theta_1)\psi_{p_3}(\Theta_2) =\\
\{X(\bar{0},m) \ass g(Y(n)), X(s(n),m) \ass Y(n)\}\{Y(n) \ass \fhat(Z(n),n), Y_1(\bar{0},\bar{0},k) \ass a\} =\\
\{X(\bar{0},m) \ass g(\fhat(Z(n),n)),\ X(s(n),m) \ass \fhat(Z(n),n),\ Y(n) \ass \fhat(Z(n),n),\ Y_1(\bar{0},\bar{0},k) \ass a\}.
\end{array}
\]
\end{examp}
\begin{theorem}\label{the.comp-s-sub-sound}
The composition of standard $s$-substitutions is sound: let $\Dcal$ be the set of parameters in $\Theta_1$ and $\Theta_2$. Then, for all $p \in \Dcal$ and $\sigma \in \Scal(p)$ we have $\sigma((\Theta_1 \circ \Theta_2)|_p)\Eval = \Theta_1[\sigma]\Theta_2[\sigma]$.
\end{theorem}
\begin{proof}
Immediate by Definition~\ref{def.comp-standard-ssub}, by Definition~\ref{def.psi-to-subst} and by Theorem~\ref{the.soundness-appl}.
\end{proof} 
It remains to extend substitution and unification to schematic formulas. We restrict our analysis to simple $\omega \iota o$-theories. Let $\Theta$ be a standard $s$-substitution and $\Dcal$ a set of states containing all parameters of the formula. We first define the application of $\psi_p(\Theta)$ to $\Fob$, the basic formulas  (see Definition~\ref{def.Fbasic}). 
\begin{definition}\label{def.appl-subs-Fob}
We assume that all formulas have parameters in $\Dcal$ and $p \in \Dcal$.
Let $\Theta$ be a standard $s$-substitution and $T^*$ be an $\omega\iota$-theory. Then
\begin{itemize}
\item If $\xi \in \FV$ then $\Ap(\Theta,\xi)|_p = \xi\psi_p(\Theta) = \xi$.
\item If $P \in \Po_k$ and $t_1,\ldots,t_k \in \Fob(T^*)$ then 
$\Ap(\Theta,P(t_1,\ldots,t_k))|_p = P(\Ap(\Theta,t_1)|_p,\ldots,$ $\Ap(\Theta,t_k)|_p)$
\item If $F \in \Fob(T^*)$ then $\Ap(\Theta,\neg F)|_p = \neg \Ap(\Theta,F)|_p$.
\item If $F_1,F_2 \in \Fob(T^*)$ then $\Ap(\Theta,F_1 \circ F_2)|_p =\Ap( \Theta,F_1)|_p \circ \Ap(\Theta,F_2)|_p$ for 
$\circ \in \{\land,\lor\}$. 
\end{itemize}
\end{definition}
The extension of an $s$-substitutions to schematic formulas  is not completely trivial. It is not immediately clear how an $s$-substitution should be applied to an expression of the form $\phat(X_1,\ldots,X_j,t_1,\ldots,t_k)$ where $\phat$ is a schematic predicate symbol and $X_1,\ldots,X_j$ are variable  classes. The solution consists in introducing a new symbol $\phat(\Theta)$ which depends on $\phat$ and $\Theta$. 
\begin{definition}\label{def.appl-subst-schem-formula}
We assume that all formulas have parameters in $\Dcal$ and $p \in \Dcal$.
Let $\Theta$ be a standard $s$-substitution. Then we define 
\begin{itemize}
\item If $\xi \in \FV$ then $\Ap(\Theta,\xi)|_p =\xi\psi_p(\Theta) = \xi$.
\item If $P \in \Po_k$ and $t_1,\ldots,t_k \in \Tiotahat(T^*)$ then $\Ap(\Theta,P(t_1,\ldots,t_k))|_p =  
P(\Ap(\Theta,t_1)|_p,\ldots,$ $\Ap(\Theta,t_k)|_p)$.
\item $\Ap(\Theta,\neg F)|_p = \neg \Ap(\Theta,F)|_p$.
\item $\Ap(\Theta,F_1 \circ F_2)|_p =\Ap( \Theta,F_1)|_p \circ \Ap(\Theta,F_2)|_p$ for 
$\circ \in \{\land,\lor\}$. 
\item if $\pi = (i,j,k)$ is a profile and $\phat \in \Phat(\pi)$, $X_l \in \Viota_i$ for $l=1,\ldots,j$, $t_1,\ldots,t_k \in \Tomega$ then 
$$\Ap(\Theta,\phat(X_1,\ldots,X_j,t_1,\ldots,t_k)|_p =  \phat(\Theta)(X_1,\ldots,X_j,t_1,\ldots,t_k)|_p.$$
where $\phat(\Theta)$ is defined as follows: 
Let us assume that 
$D(\phat)$ is of the form 
\[
\begin{array}{l}
\{\phat(X_1,\ldots X_j,n_1,\ldots,n_{k-1},s(n_k)) = \phat_S\{\xi \ass \phat(X_1,\ldots,X_j,n_1,\ldots,n_{k-1},n_k)\},\\ 
\phat(X_1,\ldots X_j,n_1,\ldots,n_{k-1},\bar{0}) = \qhat_B)\}.
\end{array}
\]
Then we define $D(\phat(\Theta))|_p)(X_1,\ldots,X_j,n_1,\ldots,n_k)|_p$ as 
\begin{small}
\[
\begin{array}{l}
\{\phat(\Theta)(X_1,\ldots,X_j,n_1,\ldots,s(n_k))|_p = \Ap(\Theta,\phat_S)|_p\{\xi \ass \phat(\Theta)(X_1,\ldots,X_j,n_1,\ldots,n_{k-1},n_k)|_p\},\\
\phat(\Theta)(X_1,\ldots,X_j,n_1,\ldots,\bar{0})|_p = \Ap(\Theta,\qhat_B)|_p \}.
\end{array}
\]
\end{small}
\end{itemize}
\end{definition}
\begin{examp}\label{ex.s-subst-schemform}
Let $(T,T',\emptyset)$ the $\omega\iota o$-theory defined in Example~\ref{ex.ssub-form-schema2}. We had the following definitions of recursive predicate- and term symbols:
\begin{eqnarray*}
D(\qhat) &=& \{ \qhat(X,s(n)) = \qhat(X,n) \lor Q(\that(X(s(n)),s(n))),\ \qhat(X,\bar{0}) = Q(X(\bar{0}))\}.\\
D(\that) &=& \{\that(x,s(n)) = h(\that(x,n)),\ \that(x,\bar{0}) = x \}.
\end{eqnarray*}
and let $\Theta = \{X(s(n)) \ass \that(X(n),n)\}$, $p = n \neq \bar{0} \land p(n) \neq \bar{0}$, and $F = \qhat(X,s(n))$.
Then 
\[
\begin{array}{l}
\Ap(\Theta,F)|_p = \Ap(\Theta,\qhat(X,s(n)))|p = \qhat(\Theta)(X,s(n))|_p, \mbox{ where }\\[1ex]
\qhat(\Theta)(X,s(n))|_p = \Ap(\Theta,\qhat(X,n) \lor Q(\that(X(s(n)),s(n))))|p = \\[1ex]
\Ap(\Theta,\qhat(X,n))|_p \lor \Ap(\Theta,Q(\that(X(s(n)),s(n))))|p =\\[1ex]
 \Ap(\Theta,\qhat(X,n))|_p \lor Q(\that(X(s(n)),s(n)))\psi_p(\Theta) =\\[1ex]

\qhat(\Theta)(X,n)|_p \lor Q(\that(X(s(n)),s(n)))\Theta = \qhat(\Theta)(X,n)|p \lor Q(\that(\that(X(n),n),s(n))).
\end{array}
\]
Note that $\qhat(\Theta)(X,n)|_p$ cannot be evaluated further unless we have a parameter assignment.
\end{examp}
\begin{theorem}\label{the.s-subst-sound-on-Tiotahat}
Schematic substitution is sound on schematic formulas: let $F$ be a formula in an $\omega \iota o$-theory, $\Theta$ be a standard 
$s$-substitution and $\sigma$ a parameter assignment. Then \\
$\sigma(F)\Eval\Theta[\sigma] = \sigma(\Ap(\Theta,F)|_p)\Eval$  for $p \in \Dcal$ and $\sigma(p)\Eval = \top$.
\end{theorem}
\begin{proof}
By Theorem~\ref{the.soundness-appl} we know that, for schematic terms $t$ in $T^*$, 
$\sigma(t)\Eval\Theta[\sigma] = \sigma(\Ap(\Theta,t)|_p)\Eval$ for $\sigma(p)\Eval = \top$. If $F$ is a formula in $T^*$ without schematic predicate symbols then,  
by Definition~\ref{def.appl-subst-schem-formula}, $\sigma(F)\Eval\Theta[\sigma] = \sigma(\Ap(\Theta,F)|_p)\Eval$ for $\sigma(p)\Eval = \top$ (here we only have homorphic extensions of the application of $\Theta$). What remains to be shown is 
$$\sigma(F)\Eval \Theta[\sigma] = \sigma(\Ap(\Theta,F)|p)\Eval \mbox{ for } \sigma(p)\Eval = \top$$
and for $F=\phat(X_1,\ldots,X_j,t_1,\ldots,t_k)$ with the corresponding definitions 
\begin{eqnarray*}
\{\phat(X_1,\ldots X_j,n_1,\ldots,n_{k-1},s(n_k)) &=& \phat_S\{\xi \ass \phat(X_1,\ldots,X_j,n_1,\ldots,n_{k-1},n_k)\},\\ 
\phat(X_1,\ldots X_j,n_1,\ldots,n_{k-1},\bar{0}) &=& \qhat_B\}.
\end{eqnarray*}
Moreover, we assume $\sigma(t_i)\Eval = \bar{\alpha}_i$ for $i=1,\ldots,k$. 

Here we proceed by the depth of recursion. Let us assume that $\phat$ is minimal and $\sigma(n_k) = \bar{0}$; then in the base case 
\[
\begin{array}{l}
\sigma(\phat(X_1,\ldots X_j,\bar{\alpha}_1,\ldots,\bar{\alpha}_{k-1},\bar{0})\Eval = \phat(X_1,\ldots,X_j,\bar{\alpha_1},\ldots,\bar{\alpha}_{k-1},\bar{0})\Eval = \sigma(\phat_B)\Eval .
\end{array}
\]
$\phat_B$ is a formula without schematic predicate symbols. In this case we have already shown that 
$$\sigma(\phat_B)\Eval\Theta[\sigma] = \sigma(\Ap(\Theta,\phat_B)|_p)\Eval \mbox{ for } \sigma(p)\Eval = \top.$$
So we obtain 
\[
\begin{array}{l}
\sigma(\phat(X_1,\ldots,X_j,n_1,\ldots,n_{k-1},\bar{0}))\Eval\Theta[\sigma] =
 \sigma(\Ap(\Theta,\phat(X_1,\ldots,X_j,n_1,\ldots,\bar{\alpha}_{k-1},\bar{0}))|p)\Eval 
\end{array}
\]
for $\sigma(p)\Eval = \top$.
Note that also $\phat(\Theta)(X_1,\ldots,X_j,n_1,\ldots,\bar{0})|_p = \Ap(\Theta,\qhat_B)|_p$.

For the step case of $\phat$ ($\bar{\alpha}_k > \bar{0}$) we have that $\phat_S$ does not contain schematic predicate symbols and we may assume that for $\sigma(n_k) = \bar{\alpha}_k$ we have 
$$\sigma(\phat(X_1,\ldots,X_j,n_1,\ldots,n_{k-1},n_k))\Eval\Theta[\sigma] = \sigma(\Ap(\Theta,\phat(X_1,\ldots,X_j,n_1,\ldots,n_{k-1},n_k)|_p)\Eval$$
for $\sigma(p)=\top$.
By 
$$\phat(X_1,\ldots X_j,n_1,\ldots,n_{k-1},s(n_k)) = \phat_S\{\xi \ass \phat(X_1,\ldots,X_j,n_1,\ldots,n_{k-1},n_k)\}$$
we obtain for $\sigma(p)=\top$
\[
\begin{array}{l}
\sigma(\phat(X_1,\ldots X_j,n_1,\ldots,n_{k-1},s(n_k)))\Eval\Theta[\sigma] =\\[1ex]
\sigma(\phat_S)\Eval\Theta[\sigma]\{\xi \ass \sigma(\phat(X_1,\ldots,X_j,n_1,\ldots,n_{k-1},n_k))\Eval\Theta[\sigma]\} = \\[1ex]
\sigma(\phat_S)\Eval\Theta[\sigma] \{\xi \ass\sigma(\Ap(\Theta,\phat(X_1,\ldots,X_j,n_1,\ldots,n_{k-1},n_k))|_p)\}\Eval = \\[1ex]
\sigma(\Ap(\Theta,\phat_S)|p)\Eval\{\xi \ass\sigma(\Ap(\Theta,\phat(X_1,\ldots,X_j,n_1,\ldots,n_{k-1},n_k))|_p)\Eval\} = \\[1ex]
\sigma(\Ap(\Theta,\phat(X_1,\ldots X_j,n_1,\ldots,n_{k-1},s(n_k))|p))\Eval.
\end{array}
\]
The case for nonminimal $\phat$ is analogous with the difference that the induction hypothesis holds for $\phat_S$, $\phat_b$ as they contain only schematic symbols smaller than $\phat$.
\end{proof}

\subsection{Schematic Unification} 

Unification is crucial to any machine oriented logic, in particular to the resolution calculus~\cite{robinson1965machine}. In Section~\ref{sec.refschemata} we will present a refutational calculus for formula schemata, generalizing the first-order resolution principle to schemata. In particular we will concentrate on unification in $\omega \iota$-theories. Our first step consists in defining the concept of schematic unification.
\begin{definition}\label{def.s-unif}
Let $T^*$ be a $\omega \iota$-theory $T^*$ and $t_1,\ldots,t_\alpha \in \Tiotahat_s(T^*)$. Then the set $\{t_1,\ldots,t_\alpha\}$ is {\em unifiable} if, for all $\sigma \in \Scal$,
$\{\sigma(t_1)\EvalTstar,\ldots,\sigma(t_\alpha)\EvalTstar\}$ is first-order unifiable. 
\end{definition}
Unfortunately, unification in $\omega \iota$-theories is undecidable as we will show below. First we will show that all primitive recursive functions can be computed in such theories by translating the constructions to simple $\omega \iota$-theories. We define a constant $\zeroI$ for zero (which is now an individual constant, not a number constant) and a one-place function symbol 
$\suc$. The successor function and the predecessor function are defined via the following schematic symbols $\shat$, $\that$, $\numhat$ (all of them having the profile $(1,1)$):
\begin{eqnarray*}
D(\numhat) &=& \{\numhat(x,s(n)) = \suc(\numhat(x,n)), \numhat(x,\bar{0}) = \zeroI,\\
D(\shat) &=& \{\shat(x,n) = \suc(\numhat(n))\},\\
D(\that) &=& \{\that(x,s(n)) = \numhat(x,n),\ that(x,\bar{0}) = \zeroI\}.
\end{eqnarray*}
It is easy to see that for $\sigma(n) = \alpha$ $\sigma(\shat(x,n))\Eval = \suc^\alpha(\zeroI)$ which represents $s^\alpha(\bar{0})$ within the individual terms. In the same way $\that(x,n)$ represents the predecessor $p$: 
\[
\begin{array}{l}
\that(x,\bar{0})\Eval = \zeroI,\\
\that(x,\bar{1})\Eval = \numhat(x,\bar{0}) = \zeroI,\\
\that(x,\bar{2})\Eval = \numhat(x,\bar{1})\Eval = \suc(x,\bar{0})\Eval = \suc(\zeroI) \mbox{ etc.}
\end{array}
\]
An easy induction argument yields that for all $\sigma(n)  = \alpha \geq \bar{2}$ $\sigma(\that(x,n))\Eval = \suc^{\alpha-1}(\zeroI)$ which represents $s^{\alpha-1}(\bar{0})$ in the individual terms. It is easy to see that the variable $x$ (which could be replaced by any first-order term) plays no role in the definition above. We illustrate the recursive definition of $+$ within the individual term schemata and show how to restrict the individual terms such that the schema evaluates to the new ''numerals''. We start with 
$$D(\oplus) = \{\oplus(x,s(n)) = \suc(\oplus(x,n)),\ \oplus(x,\bar{0}) =x\}.$$
This definition would be inappropriate as $x$ which is an individual variable cannot be evaluated under a $\sigma \in \Scal$. Instead we have to evaluate $\oplus(\shat(x,m),n)$ under $\sigma$ which gives the desired result. Using this trick we can simulate the computation of primitive recursive functions in simple $\omega \iota$-theories. 
It is easy to see that the $\omega \iota$-theories obtained in this way simulate the corresponding $\omega$-theories computing primitive recursive functions. 
We call the set of schemata $\omega \iota$ theories corresponding to the primitive recursive schemata in $T^\omega$ as $\PRLiota$. While, given a $\sigma \in \Scal$, every term in  $T^\omega$ evaluates to a numeral, the terms in $\PRLiota$ evaluate to elements in $\Numiota$ which is defined as 
\begin{itemize}
\item $\zeroI \in \Numiota$,
\item if $t \in \Numiota$ then $\suc(t) \in \Numiota$.
\end{itemize}
\begin{theorem}\label{the.sunification-undecidable}
The unification problem for schematic terms in $\PRLiota$ is undecidable.
\end{theorem} 
\begin{proof}
It is well known that the equivalence problem of primitive recursive programs is undecidable~\cite{schoning1998equivalence}. Both the schematic terms defined in $\omega$-theories and their modifications in individual term schemata define a programming language for primitive recursive functions. In particular it is undecidable whether a primitive recursive program computes the function {\it zero} (${\it zero}(n) = 0$ for all $n \in \N$). Therefore the equivalence problem  $t = \zeroI$, where $t \in \PRLiota$,  is undecidable; that means it is undecidable whether for all $\sigma \in \Scal$ $\sigma(t)\Eval = \zeroI$ (note that $\sigma(\zeroI)\Eval = \zeroI$ for all $\sigma \in \Scal$). We consider now the unification problem $M\colon \{t,\zeroI\}$. According to Definition~\ref{def.s-unif} $M$ is unifiable iff for all $\sigma \in \Scal$ the set $M(\sigma)\colon \{\sigma(t)\Eval,\zeroI\Eval\} = \{\sigma(t)\Eval,\zeroI\}$ is first-order unifiable. But as $t \in \PRLiota$ we have $\sigma(t)\Eval \in \Numiota$, i.e. $\sigma(t)\Eval$ is a normalized ground term. Therefore $M(\sigma)$ is only unifiable if $\sigma(t)\Eval = \zeroI$, and $M$ is unifiable iff for all $\sigma \in \Scal$ $\sigma(t)\Eval = \zeroI$ iff the equivalence problem $t = \zeroI$ holds. Hence the unification problem for $\PRLiota$ is undecidable.
\end{proof}
As schematic unification is not only undecidable but even $\Pi_1$-complete there exists no complete method for unification, i.e. there are always unifiers which cannot be found by algorithmic methods. However, we need a method for subclasses of the problem which does not only detect unifiability but also produces a syntactic representation of a unifier. Below we extend unification to standard terms. 
\begin{definition}\label{def.unif-standard-terms}
Let $\Tcal\colon \{t_1,\ldots,t_\alpha\}$ be a set of standard terms over the parameters $\Ncal_0\colon \{n_1,$ $\ldots,$ $n_\beta\}$ and $\Dcal$ the set of states over $\Ncal_0$. A standard $s$-substitution $\Theta$ is called a {\em  standard unifier} of $\Tcal$ if for all $p \in \Dcal$ 
$$\Ap(\Theta,t_1)|p = \ldots = \Ap(\Theta,t_\alpha)|p.$$
$\Tcal$ is called {\em uniformly standard unifiable} if there exists a standard unifier of $\Tcal$.
\end{definition}
\begin{examp}\label{ex.standard-unification}
Let $T =(T_1,\emptyset)$ be a  $\omega\iota$-theory where  $T_1 =  (\that,\{\that,\shat\},D(\{\that,\shat\}),\{x,y\},<)$ and  
\begin{itemize}
\item $\pi(\that) = (2,2)$, $\pi(\shat) = (1,1)$,
\item $x,y \in \Viota_0$,
\item $\shat < \that$,
\item $g \in \Fiota_2, h \in \Fiota_1$.
\end{itemize}
Moreover,
\begin{eqnarray*}
D(\that) &=& \{\that(x,y,n_1,s(n_2)) = g(\that(x,y,n_1,n_2),\shat(x,n_1)),\ \that(x,y,n_1,\bar{0})  = y\},\\
D(\shat) &=& \{\shat(x,s(n_1)) = h(\shat(x,n_1)),\ \shat(x,\bar{0}) = x\}.
\end{eqnarray*}
Here the set of parameters is $\{n_1,n_2\}$ and $\Dcal = \{p_1,\ldots,p_9\}$ as in Example~\ref{ex.appl-Viota-2}. Let $X,Y,Z \in \Viota_2$ and 
$$\Tcal = \{g(X(n_1,\bar{0}),\that(X(n_1,n_2),Y(n_1,n_2),n_1,n_2)), g(\shat(X(\bar{0},n_2),n_1), Z(n_1,n_2)))\}.$$
Then 
$$\Theta = \{X(n_1,\bar{0}) \ass \shat(X(\bar{0},n_2),n_1), \ Z(n_1,n_2) \ass \that(X(n_1,n_2),Y(n_1,n_2),n_1,n_2))\}$$
is a standard unifier of $\Tcal$ which is easily realizable and $\Tcal$ is uniformly standard-unifiable. Note that 
$$\Tcal'\colon  \{g(X(n_1,\bar{0}),\that(X(n_1,n_2),Y(n_1,n_2),n_1,n_2)), g(\shat(X(\bar{0},n_2),n_1), Y(\bar{0},\bar{0})))\}$$
is not standard unifiable as for $\sigma(n_1) = \sigma(n_2) = \bar{0}$ we obtain 
$$\sigma(\Tcal)\Eval = \{g(X(\bar{0},\bar{0}),\that(X(\bar{0},\bar{0}),Y(\bar{0},\bar{0}),\bar{0},\bar{0})), g(\shat(X(\bar{0},\bar{0}),\bar{0}), Y(\bar{0},\bar{0})))\}$$
which is not first-order unifiable due to occurs-check: $Y(\bar{0},\bar{0})$ occurs in the term  $\that(X(\bar{0},$ $\bar{0}),$ $Y(\bar{0},\bar{0}),$ $\bar{0},$ $\bar{0}))$. Indeed, the standard substitution 
$$\Theta'\colon \{X(n_1,\bar{0}) \ass \shat(X(\bar{0},n_2),n_1), \ Y(\bar{0},\bar{0}) \ass \that(X(n_1,n_2),Y(n_1,n_2),n_1,n_2))\}$$ 
is not a standard unifier of $\Tcal'$ as for $p_1 = n_1=\bar{0} \land n_2 = \bar{0}$ 
\[
\begin{array}{l}
\Ap(\Theta', g(X(n_1,\bar{0}),\that(X(n_1,n_2),Y(n_1,n_2),n_1,n_2)))|_{p_1} \neq  \Ap(\Theta',g(\shat(X(\bar{0},n_2), Y(\bar{0},\bar{0})))|_{p_1}.
\end{array}
\]
\end{examp}
The unification problem 
$$\Tcal'\colon  \{g(X(n_1,\bar{0}),\that(X(n_1,n_2),Y(n_1,n_2),n_1,n_2)), g(\shat(X(\bar{0},n_2),n_1), Y(\bar{0},\bar{0})))\}$$
in Example~\ref{ex.standard-unification} shows that we may not simply replace a variable expression by a term without knowing that the variable expressions are not parameter-unifiable with variable expressions (in the example above $Y(n_1,n_2)$ is parameter-unifiable with $Y(\bar{0},\bar{0})$ which generates an occurs-check). It also turns out that it is easier to define unifiers for all $p \in \Dcal$ separately than starting with a substitution and testing whether it is actually a uniform standard-unifier. The example below shows that the construction of a standard unifier cannot be obtained by using the variable expressions of the terms to be unified.
\begin{examp}\label{ex.uniform}
Let $\Tcal = \{g(X_1(n),X_1(\bar{0})),\ g(X_2(\bar{0}),X_2(p(n)))\}$. Neither 
\begin{eqnarray*}
\Theta_1 &=& \{X_1(n) \ass X_2(\bar{0}),\   X_1(\bar{0}) \ass X_2(p(n))\} \mbox{ nor }\\
\Theta_2 &=& \{X_2(\bar{0}) \ass X_1(n),\ X_2(p(n)) \ass X_1(\bar{0})\}
\end{eqnarray*}
are extended $s$-unifiers as the domain variables are parameter-unifiable. Is 
\begin{eqnarray*}
\Theta_3 &=& \{X_1(n) \ass X_2(\bar{0}),\ X_2(p(n)) \ass X_1(\bar{0})\} \mbox{ an  s-unifier? }
\end{eqnarray*}
No, it is not! Consider $\Theta_3[\sigma]$ for $\sigma(n) = \bar{0}$. Then 
\begin{eqnarray*}
\Theta_3[\sigma] &=& \{X_1(\bar{0}) \ass X_2(\bar{0}), X_2(\bar{0}) \ass X_1(\bar{0})\}
\end{eqnarray*}
which is a standard $s$-substitution but no unifier of the set 
$$\sigma(\Tcal)\Eval = \{g(X_1(\bar{0}),X_1(\bar{0})),\ g(X_2(\bar{0}),X_2(\bar{0}))\}.$$
The remedy lies in the computation of unifiers for the sets 
$$ \Tcal(p_i) = \psi_{p_i}(\Tcal) = \{\psi_{p_i}(g(X_1(n),X_1(\bar{0}))),\ \psi_{p_i}(g(X_2(\bar{0}),X_2(p(n))))\}$$
for $i=1,2,3$ and 
\begin{eqnarray*}
p_1 &=& n=\bar{0},\\
p_2 &=& n \neq \bar{0} \land p(n) = \bar{0},\\
p_3 &=& n \neq \bar{0} \land p(n) \neq \bar{0}.
\end{eqnarray*}
Here we get 
\begin{eqnarray*}
\Tcal(p_1) &=&  \{g(X_1(\bar{0}),X_1(\bar{0})),\ g(X_2(\bar{0}),X_2(\bar{0}))\},\\
\Tcal(p_2) &=& \{g(X_1(\bar{1}),X_1(\bar{0})),\ g(X_2(\bar{0}),X_2(\bar{0}))\},\\
\Tcal(p_3) &=& \{g(X_1(n),X_1(\bar{0})),\ g(X_2(\bar{0}),X_2(p(n)))\}.
\end{eqnarray*}
Let us define 
\begin{eqnarray*}
\Theta_1 &=& \{X_1(\bar{0}) \ass X_2(\bar{0})\} \mbox{ and }\\
\Theta_2 &=& \{ X_1(\bar{1}) \ass X_2(\bar{0}), X_1(\bar{0}) \ass X_2(\bar{0})\}.
\end{eqnarray*}
Obviously $\Theta_1$ is a unifier of $\Tcal(p_1)$ and $\Theta_2$ of $\Tcal(p_2)$. In case $\Tcal(p_3)$ we seem to be lost as this set of terms is the original one, but we are not!
The expression 
$$\Theta_3\colon \{X_1(n) \ass X_2(\bar{0}),\ X_1(\bar{0}) \ass X_2(p(n))\}$$
is not an extended $s$-substitution, but for $\sigma \in \Scal(p_3)$ it is: indeed, for all $\sigma \in \Scal$ such that $\sigma(n) \neq \bar{0}$ and $\sigma(p(n))\Eval \neq \bar{0}$ $\Theta_3[\sigma]$ is indeed a first-order substitution and $\Theta_3[\sigma]$ unifies 
$\sigma(\Tcal(p_3))$. 
\end{examp}
The last observation above leads us to the following definition:
\begin{definition}\label{def.standard-unif-on-states}
Consider $\Tiotahat(T^*)$ for a theory $T^*$ over the parameter set $\Ncal_\alpha = \{n_1,\ldots,n_\alpha\}$ and let $\Dcal$ be the set of all states over $\Ncal_\alpha$. An expression of the form 
$$\Theta\colon \{V_1 \ass t_1,\ldots, V_\beta \ass t_\beta\}$$
for standard variable expressions $V_i$ over $\Ncal_\alpha$ and $t_i \in \Tiotahat(T^*)$ 
is called a standard $s$-substitution over $p$ for $p \in \Dcal$ if $\Theta[\sigma]$ is a first-order substitution for all 
$\sigma \in \Scal(p)$.  $\Theta$ is called a standard $s$-unifier over $p$  of a set $\Tcal$ of terms in $\Tiotahat(T^*)$ if 
$\Theta$ $s$-unifies $\Tcal$, i.e.  for all $\sigma \in \Scal(p)\colon|\sigma(\Tcal\Theta))\Eval|=1$.
\end{definition} 
\begin{examp}\label{s-unif-over-p}
Let $\Theta_3 = \{X_1(n) \ass X_2(\bar{0}),\ X_1(\bar{0}) \ass X_2(p(n))\}$ and $p_1,p_2,p_3$ be as in 
Example~\ref{ex.uniform}. Then  $\Theta_3$ is a standard $s$-unifier of $\Tcal$ over $p_3$ for 
$$\Tcal = \{g(X_1(n),X_1(\bar{0})),\ g(X_2(\bar{0}),X_2(p(n)))\}.$$
Indeed, $\Tcal\Theta_3 = \{g(X_2(\bar{0}),X_2(p(n))\}$ and $\Theta_3$ is a standard $s$-substitution over $p_3$. 
Note that $\Theta_3$ is {\em not} a standard $s$-substitution over $p_1$. 
\end{examp} 
\begin{lemma}\label{le.standard-unif.}
Let $\Ncal_\alpha$ and $\Dcal$ as in Definition~\ref{def.standard-unif-on-states} and $p \in \Dcal$. Let $\Tcal$ be a set of standard terms such that $\psi_p(\Tcal) = \Tcal$ and $V_1,\ldots,V_\beta$ be arbitrary standard variable expressions in $\Tcal$. Then for all standard terms $t_1,\ldots,t_\beta$ in $\Tiotahat(T^*)$ for a theory $T^*$ the expression 
$$\{V_1 \ass t_1,\ldots, V_\beta \ass t_\beta\}$$
is a standard $s$-substitution if the $V_j$ are pairwise different. 
\end{lemma}
\begin{proof}
It is enough to show that for $i,j \in \{1,\ldots,\beta\}$ and $i \neq j$ $\{V_i,V_j\}$ is not parameter-unifiable. 
Let $X(r_1,\ldots,r_\gamma)$ and $X(s_1,\ldots,s_\gamma)$ be two different variable expressions occurring in $\Tcal$ and  
$\psi_p(X(r_1,\ldots,r_\gamma)) = X(r_1,\ldots,r_\gamma)$, $\psi_p(X(s_1,\ldots,s_\gamma)) = X(s_1,\ldots,s_\gamma)$. 
In the proof of Lemma~\ref{le.psi-p} we have shown that for the variable expressions above we have
\begin{itemize}
\item $X(r_1,\ldots,r_\gamma) =  X(s_1,\ldots,s_\gamma)$ iff there exists a $\sigma \in \Scal(p)$ such that 
$$\sigma(X(r_1,\ldots,r_\gamma))\Eval = \sigma(X(s_1,\ldots,s_\gamma))\Eval.$$
\end{itemize}
I.e., as $X(r_1,\ldots,r_\gamma)$  and $X(s_1,\ldots,s_\gamma)$ are different they are not parameter-unifiable.
\end{proof} 
Below we define a sound (though incomplete) unification algorithm for sets of terms in $\Tiotahat(T^*)$ for a theory $T^*$. Given a set of standard terms $\Tcal\colon \{t_1,\ldots,t_\alpha\}$ for $\alpha>1$ (to be unified) we start with the following preprocessing:
\begin{itemize}
\item Check whether there are two different variable expressions which are parameter-unifiable.
\item If there are none then define $\Tcal^* = \{\Tcal\}$.
\item If there exist different parameter-unifiable variable expressions then select the variable expression with the highest arity, e.g. 
$X(s_1,\ldots,s_\beta)$ and compute the set of all states $p \in \Dcal$ for $\Dcal$ being the set of states corresponding to $X(n_1,\ldots,n_\beta)$. Then define 
$$\Tcal^* = \{\Tcal(p) \mid p \in \Dcal\}.$$
\end{itemize}
For the set $\Tcal$ (respectively $\Tcal(p)$ for $p \in \Dcal$) we define the set 
\begin{eqnarray*}
\Ucal(\Tcal) &=& \{t_1 \simeq t_2,\ldots,t_1 \simeq t_\alpha\} \mbox{ respectively },\\
\Ucal(\Tcal(p)) &=& \{\psi_p(t_1) \simeq \psi_p(t_2),\ldots,\psi_p(t_1) \simeq \psi_p(t_\alpha)\}.
\end{eqnarray*}
Now assume that $\Ucal = \Ucal(\Tcal)$ or $\Ucal = \Ucal(\Tcal(p))$. Let $\Vexp$ be the set of all standard variable expressions. Note that we have shown that the unification problem is undecidable and even that unifiability is not recursively enumerable. Therefore any unification algorithm for standard terms will be incomplete by the nature of the problem. When we arrive at an equation of the form
$$\shat(t_1,\ldots,t_\beta,r_1,\ldots,r_\gamma) \simeq \that(t'_1,\ldots,t'_\delta,r'_1,\ldots,r'_\delta)$$
such that the terms are not syntactically equal then we stop the algorithm (to be defined) although the terms might be unifiable. The following example illustrates the problem.
\begin{examp}\label{ex.unifalg}
Let us consider the set $\Ucal$ consisting of the single equation
$$\{\fhat(x,y,m,n) \simeq \fhat(u,v,m,s(n))\}.$$
for 
\[
\begin{array}{l}
D(\fhat) =\{\fhat(x,y,m,\bar{0}) = \ghat(x,y,m),\ \fhat(x,y,m,s(n)) = h(x,y,\fhat(x,y,m,n))\},\\
D(\ghat) = \{\ghat(x,y,\bar{0}) = h(x,y,\bar{0}), \ghat(x,y,s(n)) = h(\ghat(x,y,n),x,y)\}.
\end{array}
\]
The problem $\Ucal$ is not solvable as 
\[
\begin{array}{l}
\fhat(x,y,m,\bar{0}) = \ghat(x,y,m),\ \fhat(u,v,m,\bar{1}) = h(u,v,\ghat(u,v,m)) \mbox{ and }\\
\ghat(x,y,m) \mbox{ and } h(u,v,\ghat(u,v,m)) \mbox{ are not unifiable.}
\end{array}
\]
For the following set of definitions $\Ucal$ unsolvability does not appear so ''soon'':
\[
\begin{array}{l}
D(\fhat) =\{\fhat(x,y,m,\bar{0}) = \ghat(x,y,m),\ \fhat(x,y,m,s(n)) = h(x,y,\fhat(x,y,m,n))\},\\
D(\ghat) = \{\ghat(x,y,n) = x\}.
\end{array}
\]
Indeed, 
$$ \fhat(x,y,m,\bar{0}) = x,\ \fhat(u,v,m,\bar{1}) = h(u,v,u), $$
$$ \fhat(x,y,m,\bar{1}) = h(x,y,x), \fhat(u,v,m,\bar{2}) = h(u,v,h(u,v,u)).$$
While $\{\fhat(x,y,m,\bar{0}),\fhat(u,v,m,\bar{1})\}$ is unifiable $\{\fhat(x,y,m,\bar{1}),\fhat(u,v,m,\bar{2})\}$ is not (occurs check).
On the other hand, for 
\[
\begin{array}{l}
D(\fhat) =\{\fhat(x,y,m,\bar{0}) = \ghat(x,y,m),\ \fhat(x,y,m,s(n)) = f(\fhat(x,y,m,n))\},\\
D(\ghat) = \{\ghat(x,y,n) = x\}.
\end{array}
\]
we get (informally) $\fhat(x,y,m,k) = f^k(x), \fhat(u,v,m,s(k)) = f^{k+1}(u)$ and $\Ucal$ is unifiable by the simple unifier $\Theta\colon \{x \ass f(u)\}$. But even for this simple case we would need a proof by induction that $\Theta$ is indeed a unifier.
\end{examp}

By undecidability of the problem and by the fact that unifiability is a $\Pi_1$-complete property, the only property which we can obtain is soundness. We distinguish several different types of equations.
\begin{definition}
An equation $s \simeq t$ is called 
\begin{itemize}
\item a {\em clash-equation} if the head symbols of $s$ and $t$ are different function symbols in $\Fiota$,
\item an {\em occurs-equation} if one of $s$,$t$ is a variable expression properly occurring in the other term. 
\item a {\em complex equation} if $s$ and $t$ are not variable expressions $s \neq t$ and at least one of the head symbols is in 
$\Fiotahat$.
\item an {\em admissible equation} if one of $s,t$ is a variable expression, $s \neq t$ and the variable expression does not occur in the term.
\end{itemize}
\end{definition}
Below we define the reduction relation $\to_{\Rcal}$ inspired by the Martelli-Montanari unification algorithm in~\cite{Martelli-Montanari} working on finite sets of equations $\Ucal$ but does not apply the substitution rule. Instead the algorithm is a ''hybrid'' with the algorithm defined in~\cite{Leitsch.1997} on page 66. Thereby we extend the syntax by the symbols $r,\bot$ and by the equation $r \simeq \bot$.
\begin{definition}\label{def.toRcal}
$\Ucal$ is called {\em irreducible} if either 
\begin{itemize}
\item $r \simeq \bot \in \Ucal$ or
\item $\Ucal = \emptyset$ or 
\item $\Ucal$ consists of admissible equations only.
\end{itemize}
 Otherwise $\Ucal$ is called reducible under $\to_{\Rcal}$ to be defined below.
Now assume $\Ucal$ is reducible, then 
\begin{itemize}
\item (failure) If $s \simeq t \in \Ucal$ and $s \simeq t$ is either an occurs-equation, a clash equation or a complex equation then $\Ucal \to_{\Rcal} \{r \simeq \bot\}$.
\item (remove identity) If $s \simeq s \in \Ucal$ then $\Ucal \to_{\Rcal} \Ucal \setminus \{s \simeq s\}$.
\item (decompose) If $f(s_1,\ldots,s_\alpha) \simeq f(r_1,\ldots,r_\alpha) \in \Ucal$ for $f \in \Fiota$ and $s_i,r_i \in \Tiotahat$ then 
 $$\Ucal \to_{\Rcal} \Ucal \setminus \{f(s_1,\ldots,s_\alpha) \simeq f(r_1,\ldots,r_\alpha)\} \union \{s_1 \simeq r_1,\ldots,s_\alpha \simeq r_\alpha\}.$$
\end{itemize}
If $\Ucal \to^*_{\Rcal} \Ucal'$ and $\Ucal'$ is irreducible, thus equal to $\emptyset$, equal to $\{r \simeq \bot\}$ or to a set of admissible equations, then we write $\Ucal' = \Ucal\Eval_{\Rcal}$.
\end{definition}

Now we define the unification algorithm $\UAL(\Tcal,\Theta)$, where $\Tcal$ is a finite set of terms in $\Tiotahat$ and $\Theta$ is an $s$-substitution:
\\[1ex]
\[
\begin{array}{l}
\UAL(\Tcal,\Theta) =\\
\Begin\\
\ \If\ \Ucal(\Tcal\Theta)\Eval_{\Rcal} = \emptyset\ \Then \mbox{ return } \Theta\\
\ \ \ \Else\ \If\ \Ucal(\Tcal\Theta)\Eval_{\Rcal} = \{r \simeq \bot\}\ \Then \mbox{ return } \bot\\
\ \ \ \ \Else\ \Begin\\
\ \ \ \ \ \  \mbox{ select an admissible equation } V \simeq t \mbox{ or }t \simeq V \mbox{ in } \Ucal;\\
\ \ \ \ \ \ \ \ \mbox{ return }\UAL(\Tcal\Theta\{V \ass t\},\Theta\{V \ass t\});\\
 \ \ \ \ \ \ \ \ \End ;\\
\End.
\end{array}
\]                       
\begin{lemma}\label{le.Rcal}
Let $\Ucal$ be a finite set of equations over $\Tiotahat$. Then 
\begin{itemize}
\item[(1)] $\to_{\Rcal}$ is defined on $\Ucal$ and $\Ucal\Eval_{\Rcal}$ is finite.
\item[(2)] If $\Ucal\Eval_{\Rcal} = \emptyset$ then all equations in $\Ucal$ are of the form $s \simeq s$.
\end{itemize}
\end{lemma}
\begin{proof}
(1): Assume that $\to_\Rcal$ is nonterminating on $\Ucal$. Then by Definition of $\to_{\Rcal}$ the rules (decompose) and (remove identity) must be applied infinitely often. But this is impossible as the complexity decreases under application of $\to_{\Rcal}$ via those rules. As complexity measure can choose e.g. the following one:
\begin{eqnarray*}
c(V) &=& 0 \mbox{ for variable expressions V},\\
c(f(t_1,\ldots,t_\alpha)) &=& 2^{c(t_1)+ \cdots +  c(t_\alpha)} \mbox{ for }f \in \Fiota,\\
c(\fhat(t_1,\ldots,t_\alpha,r_1,\ldots,r_\beta) &=& 2^{c(t_1)+ \cdot + c(t_\alpha)},\\
c(s \simeq t) &=& c(s)+c(t).
\end{eqnarray*}
(2) if there exists an equation $s \simeq t$ such that $s \neq t$ and $\to_{\Rcal}$ does not terminate with $r \simeq \bot$ then there must exist an equation of the form $V \simeq r$ or $r \simeq V$ in $\Ucal\Eval_{\Rcal}$ and $\Ucal\Eval_{\Rcal}$ is not empty.
\end{proof}
The unifier of a finite set of terms $\Tcal$ in $\Tiotahat$ is then computed by calling $\UAL(\Tcal,\epsilon)$ where $\epsilon$ denotes the identical substitution $\emptyset$.
\begin{theorem}\label{the.unif}
(1) For all finite sets $\Tcal$ of terms in $\Tiotahat$ such that no parameter-unifiable variable expressions exist in $\Tcal$ we have 
\begin{itemize}
\item $\UAL(\Tcal,\epsilon)$ terminates.
\item If $\UAL(\Tcal,\epsilon)$ terminates with an $s$-substitution $\Theta$ then $\Theta$ is an $s$-unifier of $\Tcal$.
\end{itemize}
(2) For all finite sets $\Tcal$ of terms with parameter-unifiable variable expressions and appropriate sets of states $\Dcal$ we get 
\begin{itemize}
\item $\UAL(\Tcal(p),\epsilon)$ terminates for all $p \in \Dcal$.
\item If, for all $p \in \Dcal$ $\UAL(\Tcal(p),\epsilon)$ terminates with an $s$-substitution $\Theta(p)$ then $\Tcal$ is unifiable and 
for $\sigma(p) = \top$ $\Theta(p)$ is a unifier of $\Tcal(p)$.
\end{itemize} 
\end{theorem}
\begin{proof}
We prove a more general result in replacing $\UAL(\Tcal,\epsilon)$ by $\UAL(\Tcal,\Theta)$ for an arbitrary $\Theta$. In every call of 
$\UAL(\Tcal,\Theta)$ we first compute $\Ucal(\Tcal\Theta)\Eval_{\Rcal}$, where the computation terminates and the result is finite. If  
$\Ucal(\Tcal\Theta)\Eval_{\Rcal} = \emptyset$ the the algorithm returns a substitution $\Theta$. If $\Ucal(\Tcal)\Eval_{\Rcal} = \{r \simeq \bot\}$ then $\UAL$ also terminates with $\bot$. If both cases do not apply then we call 
$$\UAL(\Tcal\Theta\{V \ass t\},\Theta\{V \ass t\}).$$
But, as $t$ does not contain $V$ the set of terms $\Tcal\Theta\{V \ass t\}$ contains one variable expression less. Therefore the number of calls is limited by the number of variable expressions in the input set $\Tcal$. Thus $\UAL$ is always terminating. This holds also for case (2) where we only have several different problems of the same type.\\[1ex]
It remains to show that, in case $\UAL$ returns a substitution $\Theta$, $\Theta$ is a unifier of $\Tcal$. We proceed by induction on the number $\gamma$ of calls of $\UAL$.
\begin{itemize}
\item $\gamma = 1$: Then $\Ucal(\Tcal\Theta)\Eval_{\Rcal} = \Theta$ and $\UAL$ terminates with $\Theta$. $\Theta$ is indeed a unifier of $\Tcal$ as by Lemma~\ref{le.Rcal} all equations in $\Ucal(\Tcal\Theta)$ are of type $s=s$ which implies that all terms in $\Tcal$ are equal.
\item (IH) Assume that, for all $\Tcal$ and $\Theta$ where $\UAL(\Tcal,\Theta)$ terminates after $\leq \gamma$ calls and yields a substitution $\Theta'$, this $\Theta'$ is a unifier of  $\Tcal\Theta$.\\[1ex]
Now let $\Tcal\Theta$ be a set of terms where $\UAL(\Tcal,\Theta)$ terminates after $\gamma+1$ steps and returns $\Theta'$. Then by definition of $\UAL$ there exists a variable expression $V$ and a term $t$ that 
$$\UAL(\Tcal,\Theta) = \UAL(\Tcal\Theta\{V \ass t\},\Theta\{V \ass t\}).$$
Now $\UAL(\Tcal\Theta\{V \ass t\},\Theta\{V \ass t\}$ must terminate after $\gamma$ steps and return $\Theta''$. By (IH), $\Theta''$ is a unifier of $\Tcal\Theta\{V \ass t\}$. Hence $\Theta' = \Theta\{V \ass t\}\Theta''$ and  $\Theta'$ is a unifier of 
$\Tcal\Theta$. 
\end{itemize}
Choosing $\Theta = \epsilon$ yields (1).  \\[1ex]
In case (2) we must get a unifier {\em for all} $p \in \Dcal$ as the $\Dcal$ define a partition of the parameter space. Assume that for some $p \in \Dcal$ $\UAL(\Tcal(p),\epsilon)$ terminates with $\bot$. As $\Scal'\colon \{\sigma \mid \sigma \in \Scal, \sigma(p)\Eval = \top\} \neq \emptyset$ either $\Tcal(p)\sigma$ for $\sigma \in \Scal'$ is not unifiable at all or no unifier can be found.
\end{proof}
\begin{examp}
Let $$\Tcal = \{f(\fhat(X(\bar{0},m),X(s(n),m),n,m),Y(\bar{0})), f(Z(\bar{0}),\ghat(Z(s(n)),n))\}$$  
for $f \in \Fiota$ and $\fhat,\ghat \in \Fiotahat$.
No pair in the set of  variable expressions 
$$\{X(\bar{0},m),X(s(n),m),Y(\bar{0}),Z(\bar{0}),Z(s(n))\}$$ 
is parameter-unifiable and so we can apply 
$\UAL$ to $\Tcal$ only. For better legibility we use the variable names  $x,y,z,u,v$ and write 
$$\Tcal = \{f(\fhat(x,y,n,m),z), f(u,\ghat(v,n))\}.$$ 
We now compute $\UAL(\Tcal,\epsilon)$: 
\begin{itemize}
\item $\Ucal(\Tcal) = \{f(\fhat(x,y,n,m),z) \simeq f(u,\ghat(v,n))\}$.
\item $\Ucal(\Tcal)\Eval_{\Rcal} = \{\fhat(x,y,n,m) \simeq u, z \simeq \ghat(v,n)\}$.
\item We choose the first equation, define $\Theta = \epsilon\{u \ass \fhat(x,y,n,m)\} =\{u \ass \fhat(x,y,n,m)\}$ and call 
$$\UAL(\Tcal\Theta,\Theta).$$
\item $\Ucal(\Tcal\Theta) = \{f(\fhat(x,y,n,m),z) \simeq f(\fhat(x,y,n,m),\ghat(v,n))\}$.
\item $\Ucal(\Tcal\Theta)\Eval_{\Rcal} = \{z \simeq \ghat(v,n)\}$. 
\item Now $\Theta \ass \Theta\{z \ass \ghat(v,n)\} = \{u \ass \fhat(x,y,n,m), z \ass \ghat(v,n)\}$  and we call again 
$$\UAL(\Tcal\Theta,\Theta).$$
Now $\Ucal(\Tcal\Theta) = \{f(\fhat(x,y,n,m),\ghat(v,n)) \simeq f(\fhat(x,y,n,m),\ghat(v,n)) \}$ and 
\item $\Ucal(\Tcal\Theta)\Eval_{\Rcal} = \emptyset$. So we return 
$$\{u \ass \fhat(x,y,n,m), z \ass \ghat(v,n)\}$$
which is indeed an $s$-unifier of $\Tcal$.
\end{itemize}
\end{examp}
In the resolution calculus to be defined in Section~\ref{sec.refschemata} we need an extension of $\UAL$ to formulas. For simplicity we will extend $\UAL$ only to quantifier-free formulas without schematic predicate symbols (but atoms may contain schematic terms, of course). On such formulas the unfication problem reduces to the unification problem of atoms of the form $P(t_1,\ldots,t_n)$ where $P$ is a nonschematic predicate symbol. If - in a set of atoms  $\Acal$ - not all atoms start with the same predicate symbol the set $\A$ is not unifiable. Otherwise consider 
$$\Acal = \{P(t^1_1,\ldots,t^1_k),\ldots,P(t^j_1,\ldots,t^j_k)\}$$
$\Acal$ is unifiable iff the set of equations $\Ucal$
\[
\begin{array}{l}
\{t^1_1 \simeq t^2_1,\ldots,t^1_1 \simeq t^j_1,\\
\ \ t^1_2 \simeq t^2_2,\ldots,t^1_2 \simeq t^j_2,\\
\ldots \\
\ldots\\
\ \ t^1_k \simeq t^2_k,\ldots,t^1_2 \simeq t^j_k\}
\end{array}
\]
is unfiable. Therefore, to unify $\Acal$, we apply $\UAL$ to $\Ucal$.
 
\section{The Resolution Calculus $\RPLnull$}\label{sec.ref-RPLnull}

The basis for the schematic refutational calculus is the calculus $\RPLnull$ for quantifier-free formulas, as introduced in \cite{CLL.2021}. This calculus combines dynamic normalization rules {\`a} la Andrews \cite{DBLP:journals/jacm/Andrews81} with the resolution rule, but in contrast to~\cite{DBLP:journals/jacm/Andrews81} does not restrict the resolution rule to atomic formulas.

The main motivation of the calculus $\RPLnull$ is that it can be extended to a schematic setting in a straightforward way, and that it is particularly suited for the extraction of Herbrand substitutions in the form of a Herbrand system of the schematic refutations.

The set of quantifier-free formulas in predicate logic will be denoted as $\PLnull$, and for simplicity we omit $\impl$, but can represent it by $\neg$ and $\lor$ in the usual way. In this setting, as sequents we consider objects of the form $\Gamma \seq \Delta$, where $\Gamma$ and $\Delta$ are multisets of formulas in $\PLnull$.

\begin{definition}[$\RPLnull$] \label{def.RPLnull}
The axioms of $\RPLnull$ are sequents $\seq F$ for $F \in \PLnull$.\\[1ex]
The rules are elimination rules for the connectives and the resolution rule.
\[
\infer[\AndrI]{\Gamma \seq \Delta, A}
   {\Gamma \seq \Delta, A \land B} \ \
\infer[\AndrII]{\Gamma \seq \Delta, B}
   {\Gamma \seq \Delta, A \land B} \ \
\infer[\Andl]{A,B,\Gamma \seq \Delta}
  { A \land B,\Gamma \seq \Delta
  }
\]
\[
\infer[\Orr]{\Gamma \seq \Delta,A,B}
 { \Gamma \seq \Delta, A \lor B} \ \
\infer[\OrlI]{A,\Gamma \seq \Delta}
 { A \lor B,\Gamma \seq \Delta} \ \ 
\infer[\OrlII]{B,\Gamma \seq \Delta}
 { A \lor B,\Gamma \seq \Delta} 
\]
\[
\infer[\negr]{A, \Gamma \seq \Delta}
 {\Gamma \seq \Delta, \neg A} \ \  
\infer[\negl]{\Gamma \seq \Delta,A}
 { \neg A, \Gamma \seq \Delta}
\]
The resolution rule, where $\sigma$ is an m.g.u. of $\{A_1,\ldots,A_k,B_1,\ldots,B_l\}$ and\\ 
$V(\{A_1,\ldots,A_k\}) \intrs V(\{B_1,\ldots,B_l\}) = \emptyset$ is 
\[
\infer[\res]{\Gamma\sigma,\Pi\sigma \seq \Delta\sigma,\Lambda\sigma} 
{\Gamma \seq \Delta,A_1,\ldots,A_k 
  &
  B_1,\ldots,B_m,\Pi \seq \Lambda
}
\]
A $\RPLnull$-derivation is a tree formed from axioms $\seq F$ for $F \in \PLnull$ by application of the rules above.
\end{definition}
In \cite{CLL.2021} $\RPLnull$ is shown to be sound and refutationally complete. 
\begin{examp} \label{ex.runningRefutation}
Let 
$$F = Pa \land (Pf^{2}a \lor \neg Pf^{2}a) \land \neg Pf^{4}a \land (\neg P\alpha \lor Pf^{2}\alpha),$$
$$H = (Pf^{2}a \lor \neg Pf^{2}a) \land \neg Pf^{4}a \land (\neg P\alpha \lor Pf^{2}\alpha),$$
$$G = \neg Pf^{4}a \land (\neg P\alpha \lor Pf^{2}\alpha).$$
We obtain the $\RPLnull$-refutation $\rho = $
\begin{prooftree}
	\AxiomC{$(\rho_1)$}
	\noLine
	\UnaryInfC{$Pf^2a \vdash $}
	
	\AxiomC{$\vdash F$}
	\RightLabel{$\land_{r_2}$}
	\UnaryInfC{$\vdash H$}
	\RightLabel{$\land_{r_2}$}
	\UnaryInfC{$\vdash G$}
	\RightLabel{$\land_{r_2}$}
	\UnaryInfC{$\vdash \neg P\alpha \lor Pf^{2}\alpha$}
	\RightLabel{$\lor_{r}$}
	\UnaryInfC{$\vdash \neg P\alpha, Pf^{2}\alpha$}
	\RightLabel{$\neg_r$}
	\UnaryInfC{$P\alpha \vdash Pf^{2}\alpha$}
	
	\RightLabel{$res \{\alpha \leftarrow a\}$}
	\BinaryInfC{$Pa \vdash $}	
	
	\AxiomC{$\vdash F$}
	\RightLabel{$\land_{r_1}$}
	\UnaryInfC{$\vdash Pa$}
	\RightLabel{$res$}
	\BinaryInfC{$\vdash$}
\end{prooftree}
where $\rho_1 = $
\begin{prooftree}
	\AxiomC{$\vdash F$}
	\RightLabel{$\land_{r_2}$}
	\UnaryInfC{$\vdash H$}
	\RightLabel{$\land_{r_2}$}
	\UnaryInfC{$\vdash G$}
	\RightLabel{$\land_{r_1}$}
	\UnaryInfC{$\vdash \neg Pf^{4}a$}
	\RightLabel{$\neg_r$}
	\UnaryInfC{$Pf^{4}a \vdash$}
	\AxiomC{$\vdash F$}
	\RightLabel{$\land_{r_2}$}
	\UnaryInfC{$\vdash H$}
	\RightLabel{$\land_{r_2}$}
	\UnaryInfC{$\vdash G$}
	\RightLabel{$\land_{r_2}$}
	\UnaryInfC{$\vdash \neg P\alpha \lor Pf^{2}\alpha$}
	\RightLabel{$\lor_{r}$}
	\UnaryInfC{$\vdash \neg P\alpha, Pf^{2}\alpha$}
	\RightLabel{$\neg_r$}
	\UnaryInfC{$P\alpha \vdash Pf^{2}\alpha$}
	\RightLabel{$res \{\alpha \leftarrow f^2a\}$}
	\BinaryInfC{$Pf^2a \vdash $}
\end{prooftree}
The set of unifiers in $\rho$ is $\{\{\alpha \leftarrow a\}, \{\alpha \leftarrow f^2a\}\}$.
\end{examp}
As shown in the example above, several resolution rules may occur in a $\RPLnull$-derivation, and hence several most general unifiers $\sigma_i$ need to be applied. A total unifier (or total m.g.u.) can be obtained by considering the most general simultaneous unifier of the unification problems given by the atoms in the premises of all resolution rules, after regularizing the derivation by renaming variables.
\begin{examp} \label{ex.runningRefutationRegular}
Let $\rho$ be the $\RPLnull$-derivation in Example \ref{ex.runningRefutation}, then $\rho_r =$
\begin{prooftree}
	\AxiomC{$(\rho_{r_1})$}
	\noLine
	\UnaryInfC{$Pf^2a \vdash $}
	
	\AxiomC{$\vdash F\{\alpha \leftarrow \beta\}$}
	\RightLabel{$\land_{r_2}$}
	\UnaryInfC{$\vdash H\{\alpha \leftarrow \beta\}$}
	\RightLabel{$\land_{r_2}$}
	\UnaryInfC{$\vdash G\{\alpha \leftarrow \beta\}$}
	\RightLabel{$\land_{r_2}$}
	\UnaryInfC{$\vdash \neg P\beta \lor Pf^{2}\beta$}
	\RightLabel{$\lor_{r}$}
	\UnaryInfC{$\vdash \neg P\beta, Pf^{2}\beta$}
	\RightLabel{$\neg_r$}
	\UnaryInfC{$P\beta \vdash Pf^{2}\beta$}
	
	\RightLabel{$res \{\beta \leftarrow a\}$}
	\BinaryInfC{$Pa \vdash $}	
	
	\AxiomC{$\vdash F$}
	\RightLabel{$\land_{r_1}$}
	\UnaryInfC{$\vdash Pa$}
	\RightLabel{$res$}
	\BinaryInfC{$\vdash$}
\end{prooftree}
where $\rho_{r_1} = $
\begin{prooftree}
	\AxiomC{$\vdash F\{\alpha \leftarrow \delta\}$}
	\RightLabel{$\land_{r_2}$}
	\UnaryInfC{$\vdash H\{\alpha \leftarrow \delta\}$}
	\RightLabel{$\land_{r_2}$}
	\UnaryInfC{$\vdash G\{\alpha \leftarrow \delta\}$}
	\RightLabel{$\land_{r_1}$}
	\UnaryInfC{$\vdash \neg Pf^{4}a$}
	\RightLabel{$\neg_r$}
	\UnaryInfC{$Pf^{4}a \vdash$}
	\AxiomC{$\vdash F\{\alpha \leftarrow \gamma\}$}
	\RightLabel{$\land_{r_2}$}
	\UnaryInfC{$\vdash H\{\alpha \leftarrow \gamma\}$}
	\RightLabel{$\land_{r_2}$}
	\UnaryInfC{$\vdash G\{\alpha \leftarrow \gamma\}$}
	\RightLabel{$\land_{r_2}$}
	\UnaryInfC{$\vdash \neg P\gamma \lor Pf^{2}\gamma$}
	\RightLabel{$\lor_{r}$}
	\UnaryInfC{$\vdash \neg P\gamma, Pf^{2}\gamma$}
	\RightLabel{$\neg_r$}
	\UnaryInfC{$P\gamma \vdash Pf^{2}\gamma$}
	\RightLabel{$res \{\gamma \leftarrow f^2a\}$}
	\BinaryInfC{$Pf^2a \vdash $}
\end{prooftree}
is the $\RPLnull$-derivation obtained by regularizing $\rho$.
The set of substitutions in $\rho_r$ is $\{\{\beta \leftarrow a\}, \{\gamma \leftarrow f^2a\}\}$.
\end{examp}
\begin{definition}[simultaneous unifier]
Let $W = (\Acal_1, \ldots, \Acal_n)$, where the $\Acal_i$ are nonempty sets of atoms for $i = 1, \ldots, n$. A substitution $\sigma$ is called a simultaneous unifier of $W$ if $\sigma$ unifies all $\Acal_i$. $\sigma$ is called a most general simultaneous unifier of $W$ if $\sigma$ is a simultaneous unifier of $W$ and $\sigma \leq_s \sigma'$ for all simultaneous unifiers $\sigma'$ of $W$.
\end{definition}
\begin{definition}[total m.g.u.]
Let $\rho_r$ be a regular $\RPLnull$-derivation containing $n$ resolution inferences. The unification problem of $\rho_r$ is defined as $W = (\Acal_1, \ldots, \Acal_n)$, where $\Acal_i$ ($i \in \{1, \ldots n\}$) is the set $\{A_1^i,\ldots,A_{k_i}^i,B_1^i,\ldots,B_{m_i}^i\}$ in the resolution inferences 
\[
\infer[\res]{\Gamma\sigma_i,\Pi\sigma_i \seq \Delta\sigma_i,\Lambda\sigma_i} 
{\Gamma \seq \Delta,A_1^i,\ldots,A_{k_i}^i 
  &
  B_1^i,\ldots,B_{m_i}^i,\Pi \seq \Lambda
}
\]
in $\rho_r$. If $\sigma$ is a most general simultaneous unifier of $W$, $\sigma$ is called a total m.g.u. of $\rho_r$.
\end{definition}
\begin{examp}
The total m.g.u. of $\rho_r$ in Example \ref{ex.runningRefutationRegular} is $\{\{\beta \leftarrow a\}, \{\gamma \leftarrow f^2a\}\}$.
\end{examp}

\section{Refutation Schemata}\label{sec.refschemata}

In this section we will extend $\RPLnull$ by rules handling schematic formula definitions. In inductive proofs the use of lemmas is vital, i.e. an ordinary refutational calculus, which has just a weak capacity of lemma generation, may fail to derive the desired invariant. Therefore, we will add introduction rules for the connectives, giving us the potential to derive more complex formulas. Furthermore, we have to ensure that the formulas on which the resolution rule is applied have pairwise disjoint variables. We need a corresponding concept of disjointness for the schematic case.
\begin{definition}
Let $A, B$ be finite sets of schematic variables. $A$ and $B$ are called essentially disjoint if for all $\sigma \in \Scal$ $\sigma(A)\Eval  \cap \sigma(B)\Eval = \emptyset$. 
\end{definition}
\begin{definition}[$\RPLnull^\Psi$]\label{def.RPLnullpsi}
Let $\Psi\colon ((\qhat, \Pcal, D(\Pcal), \vec{Y}, <, T', T'')$ be a theory, then for all schematic predicate symbols $\phat \in \Pcal$ for 
$D(\phat)(\vec{Y}, \vec{n},0) = \phat_B,$ and $D(\phat)(\vec{Y}, \vec{n},s(m)) = \phat_S \{\xi \ass \phat(\vec{Y},\vec{n},p(m))\}$, 
we define the elimination of defined symbols
\[
\infer[B\phat r]{\Gamma \vdash \Delta, \phat_B}
  {\Gamma \vdash \Delta, \phat(\vec{Y},\vec{n},0)
  }
\ \ 
\infer[S\phat r]{\Gamma \vdash \Delta,\phat_S\{\xi \ass \phat(\vec{Y},\vec{n},m)\} }
 { \Gamma \vdash \Delta,\phat(\vec{Y},\vec{n},s(m))
 }
\]
\[
\infer[B\phat l]{\phat_B, \Gamma \vdash \Delta}
  {  \phat(\vec{Y},\vec{n},0),\Gamma \vdash \Delta
  }
\ \ 
\infer[S\phat l]{\phat_S\{\xi \ass \phat(\vec{Y},\vec{n},m)\}, \Gamma \vdash \Delta}
 {\phat(\vec{Y},\vec{n},s(m)) ,\Gamma \vdash \Delta
 }
\]
and the introduction of defined symbols
\[
\infer[B\phat r^+]{\Gamma \vdash \Delta, \phat(\vec{Y},\vec{n},0)}
  {\Gamma \vdash \Delta,\phat_B}
\ \ 
\infer[S\phat r^+] { \Gamma \vdash \Delta,\phat(\vec{Y},\vec{n},s(m))}
 {\Gamma \vdash \Delta,\phat_S\{\xi \ass \phat(\vec{Y},\vec{n},m)\} }
\]
\[
\infer[B\phat l^+]{  \phat(\vec{Y},\vec{n},0),\Gamma \vdash \Delta} 
{\phat_B, \Gamma \vdash \Delta} 
\ \ 
\infer[S\Phat l^+] {\phat(\vec{Y},\vec{n},s(m)) ,\Gamma \vdash \Delta} 
{\phat_S\{\xi \ass \phat(\vec{Y},\vec{n},m)\}, \Gamma \vdash \Delta} 
\]
As the theory $\Psi$ contains the $\omega \iota$-theory $T'$ and, due to the defining equations for schematic term symbols, $T'$ is an equational theory we also need inference rules {\em within} schematic formulas. Let $t$ be a term in $T'$ of the form $(1)\colon \shat(s_1,\ldots,s_i,t_1,\ldots,t_{j-1},s(w))$ or of the form  $(2)\colon \shat(s_1,\ldots,s_i,t_1,\ldots,t_{j-1},\bar{0})$ where $D(\shat)=$
\[
\begin{array}{l}
\{\shat(x_1,\ldots,x_i,n_1,\ldots,n_{j-1},s(n_j)) = \shat_S\{z \ass \shat(x_1,\ldots,x_i,n_1,\ldots,n_{j-1},n_j)\},\\
 \shat(x_1,\ldots,x_i,n_1,\ldots,n_{j-1},\bar{0}) =\shat_B\}
\end{array}
\]
Then, in case (1) we define $t \sim_{T'} t'$ for 
$$t' = \shat_S\theta\{z \ass \shat(s_1,\ldots,s_i,t_1,\ldots,t_{j-1},w)\},$$
where $\theta = \{x_1 \ass s_1,\ldots,x_i \ass s_i, n_1 \ass t_1,\ldots,n_{j-1} \ass t_{j-1},n_j \ass w\}$. \\[1ex]
In case (2) we define $t \sim_{T'} t'$ for 
$$t' = \shat_B\{x_1 \ass s_1,\ldots,x_i \ass s_i, n_1 \ass t_1,\ldots,n_{j-1} \ass t_{j-1}\}.$$
We also define $\sim_{T'}$ as the reflexive and symmetric closure of the relation defined above. Now let $F$ be a schematic formula occurring in $\Psi$ such that $t$ occurs at position $\lambda$ in $F$ and $t \sim_{T'} t'$. Then $F[t']_\lambda$ is called a {\em $T'$-variant of $F$}. Now let 
$S\colon F_1,\ldots,F_i \seq G_1,\ldots,G_j$ be a sequent, the $F'_\alpha$ be $T'$-variants of $F_\alpha$ and the $G'_\beta$ be $T'$-variants of $G_\beta$. Then we define the inference 
\[
\infer[T']{F'_1,\ldots,F'_i \seq G'_1,\ldots,G'_j}
  {   F_1,\ldots,F_i \seq G_1,\ldots,G_j
  }
\]
 
We also adapt the resolution rule to the schematic case: \\
Let $T^\iota_V(\{A_1,\ldots,A_\alpha\}), T^\iota_V(\{B_1,\ldots,B_\beta\})$ be essentially disjoint sets of schematic variables and $\Theta$ be an s-unifier of $\{A_1,\ldots,A_\alpha,B_1,\ldots,B_\beta\}$. Then the resolution rule is defined as  
\[
\infer[res\{\Theta\}]{\Gamma\Theta,\Pi\Theta \vdash \Delta\Theta,\Lambda\Theta} 
{\Gamma \vdash \Delta,A_1,\ldots,A_\alpha
  &
  B_1,\ldots,B_\beta,\Pi \vdash \Lambda
}
\]
\end{definition}
The refutational completeness of $\RPLnull^\Psi$ is not an issue as already $\RPLnull$ is refutationally complete for $\PLnull$ formulas \cite{VAravantinos2013,LPW17}. Note that this is not the case anymore if parameters occur in formulas. Indeed, due to the usual theoretical limitations, the logic is not semi-decidable for schematic formulas~\cite{VAravantinos2011}. $\RPLnull^\Psi$ is sound.
\begin{proposition}\label{prop.RPLnull-sound}
Let the sequent $S$ be derivable in $\RPLnull^\Psi$ for $\Psi\colon (\qhat, \Pcal, D(\Pcal), \vec{Y}, <), T', T'')$. Then $D(\Pcal) \models S$.
\end{proposition}
\begin{proof}
The introduction and elimination rules for defined predicate symbols are sound (we have to consider the equations $D(\Pcal)$), as are the $T'$ rules; also the resolution rule (involving s-unification ) is sound.
\end{proof} 
Before giving a formal definition of refutation schemata let us have a look at the following example. The formula schema below is a characteristic formula schema of a proof schema $\Phi$ defined in~\cite{Thesis.Lolic.2020}; thus its refutation is a key step in the analysis of $\Phi$ via the schematic CERES method. 
\begin{examp}\label{ex.proofschema}
We define a refutation schema for the theory 
$$\Psi : ((\qhat, \Pcal, D(\Pcal)), \vec{Y}, <), T', T''),$$ 
$\Pcal : \{\phat, \qhat, \fhat\}$, where $\vec{Y} = (X,Y,Z)$ and
\begin{eqnarray*}
D(\phat) = \{\phat(X,s(n)) & = & \phat(X,n) \lor \neg P(X(s(n)),\fhat(a,s(n))), \\
 \phat(X,\overline{0}) & = & \neg P(X(\overline{0}),\fhat(a,\overline{0}))\}
\end{eqnarray*}
$$ D(\qhat) = \{\qhat(X,Y,Z,n,m) = P(\fhat(Y(n),m),Z(n)) \land \phat(X,n)\}. $$
We extend the calculus $\RPLnull^\Psi$ by two rules involving link-variables. We create link variables 
$V_{\phat}(X,r)$ corresponding to the formulas $\phat(X,r)$ for $r \in \{\bar{0},n,p(n)\}$ and add rules for the elimination and the introduction of the variable $V_{\phat}(X,r)$:
\[
\infer[V_{\phat}E]{\vdash \phat(X,r)}{V_{\phat}(X,r)} \ \ \infer[V_{\phat}I]{V_{\phat}(X,r)}{\vdash \phat(X,r)}
\]
The idea is to use the link variables to define recursive proofs. We start with a proof which is defined for $n>\bar{0}$ as $\rho(X,Y,Z,n,m,V_{\phat}(X,n)) =$
\[
\infer[V_{\phat}I]{V_{\phat}(X,p(n))}
 { \infer[R(\Theta(n,m))]{\vdash \phat(X,p(n))}
    { \infer{P(X(n),\fhat(a,n)) \seq \phat(X,p(n))}
         { \infer{\seq \phat(X,p(n)), \neg P(X(n),\fhat(a,n))}
            { \infer{\seq \phat(X,p(n)) \lor \neg P(X(n),\fhat(a,n))}
               { \infer[V_{\phat}E]{\seq \phat(X,n)}{V_{\phat}(X,n)}
               }
             }
          }
    &
       \seq P(\fhat(Y(n),m),Z(n))
      }
}
\]
where $\Theta(n,m) = \{X(n) \leftarrow \fhat(Y(n),m), Z(n) \leftarrow \fhat(a,n)\}$ is an s-substitution.
The idea is now to append the proof above to itself until we arrive at the sequent $\vdash \phat(X,\bar{0})$. We achieve this by the following recursive definition (where $\rhohat$ is a recursive proof symbol).
\[
\begin{array}{l}
\rhohat(X,Y,Z,n,m,V_{\phat}(X,n)) =\ \If\ n=\bar{0}\ \Then\ \seq \phat(X,\bar{0})\\
\mbox{   }\Else\ \rho(X,Y,Z,n,m,V_{\phat}(X,n)) \circ \rhohat(X,Y,Z,p(n),m,V_{\phat}(X,p(n))).
\end{array}
\]
In carrying out the composition $\rho(X,Y,Z,n,m,V_{\phat}(X,n)) \circ \rhohat(X,Y,Z,p(n),m,V_{\phat}(X,p(n)))$ we identify the last sequent of $\rho(X,Y,Z,n,m,V_{\phat}(X,n))$ with the uppermost leaf 
$\seq V_{\phat}(X,p(n))$ in $\rhohat(X,Y,Z,p(n),m,V_{\phat}(X,p(n)))$ provided $p(n)>\bar{0}$, otherwise we end up with the end-sequent $\seq \phat(X,\bar{0})$. Still we do not have a proof with the right axioms, as - for $n > \bar{0}$ one axiom in 
$\rhohat(X,Y,Z,n,m,V_{\phat}(X,n))$ is $\seq V_{\phat}(X,n)$. We only have to apply the substitution 
$\{V_{\phat}(X,n) \ass \ \vdash \phat(X,n)$ to achieve the proof 
$$\rhohat(X,Y,Z,n,m,\phat(X,n))$$
which is a proof of $\seq \phat(X,\bar{0})$ from the axioms $\seq \phat(X,n)$ and $\seq P(\fhat(Y(k),m),Z(k))$ for $k \leq n$. As both $\seq \phat(X,n)$ and the sequents $\seq P(\fhat(Y(k),m),Z(k))$ are derivable from $\vdash \qhat(X,Y,Z,n,m)$ (we also use variable renaming) the following derivation below is a refutation of 
$\seq \qhat(X,Y,Z,n,m)$:
$\rho_0(X,Y,Z,n,m)=$
\[
\infer[R(\Theta(\bar{0},m))]{\seq}
 { \infer{P(X(\bar{0}),f(a,\bar{0})) \seq}
    { \infer{ \seq \neg P(X(\bar{0}),\fhat(a,\bar{0}))}
      {\deduce{\vdash \phat(X,\bar{0})}{\rhohat(X,Y,Z,n,m,\phat(X,n))}
       }
    }
   & 
   \seq P(\fhat(Y(0),m),Z(0))
 }
\]
where $\Theta(\bar{0},m) = \{X(\bar{0}) \leftarrow \fhat(Y(0),m), Z(0) \leftarrow \fhat(a,\bar{0})\}$ is an s-substitution.
We can also explicitly insert the missing proofs from $\seq \qhat(X,Y,Z,n,m)$:\\[1ex]
To obtain  a derivation of the leaf $\seq \phat(X,n)$ we just define 
$$\rhohat(X,Y,Z,n,m,V_{\phat}(X,n))\{V_{\phat}(X,n) \ass \rho'\}$$
for $\rho'=$
\[
\infer{V_{\phat}(X,n)}
{\infer{\seq \phat(X,n)}
  { \infer{\seq P(\fhat(Y(n),m),Z(n)) \land \phat(X,n)}
     { \seq \qhat(X,Y,Z,n,m)
      }
  }
}
\]
The other proof is just 
\[
\infer{\seq P(\fhat(Y(0),m),Z(0)) }
  { \infer{\seq P(\fhat(Y(0),m),Z(0)) \land \phat(X,n)}
     { \seq \qhat(X,Y,Z,n,m)
      }
  }
\]
\end{examp}

\begin{definition}[link-variables]\label{def.proofV}
Let $\phat$ be a schematic predicate symbol. To $\phat$ we assign an infinite set of {\em link variables} $V(\phat) = \{V_i \mid i \in \N\}$; for different schematic predicate symbols the link variables are disjoint. If $V \in V(\phat)$ we also say that $V$ is of type 
$\phat$. Let $\phat(\vec{X},\vec{r})$ be a schematic atom defined via $\phat$. Then, for every $V \in V(\phat)$, the expression $V(\vec{X},\vec{r})$ is called a {\em link expression} corresponding to $\phat$; we also write $V_{\phat}(\vec{X},\vec{r})$ for this link expression to emphasize that $V$ is in $V(\phat)$.    
Two link expressions $V_{\phat}(\vec{X},\vec{r})$ and $U_{\qhat}(\vec{Y},\vec{s})$ are defined as equal if $U=V$, $\vec{X} = \vec{Y}$ and $\vec{r} = \vec{s}$. 
\end{definition}
Link variables $V$ serve the purpose to define locations in a proof where $V$ can be replaced by a proof; these locations  can be either the leaves of a proof or the root. In order to place into or to  remove variables from proofs we extend the  
$\RPLnull^\Psi$-calculus by variable elimination rules  and variable introduction rules.
\begin{definition}[$\RPLnull^\Psi V$]\label{def.extended-RPLnull-Psi}
The calculus $\RPLnull^\Psi V$ contains the rules of $\RPLnull^\Psi$ with two additional rules. Let $V$ be a variable of type 
$\phat$ then we define the rules 
\[
\infer[V_I]{V(\vec{X},\vec{r})}{\seq \phat(\vec{X},\vec{r})}  \ \ \infer[V_E]{\seq \phat(\vec{X},\vec{r})}{V(\vec{X},\vec{r})} 
\]
Elimination rules can only be applied to leaves in an $\RPLnull^\Psi V$ -derivation (if the leaf is a link expression) and introduction rules to root nodes which are labeled by a sequent of the form $\seq \phat(\vec{X},\vec{r})$. Or expressed in another way: any 
$\RPLnull^\Psi V$-derivation can be obtained from an $\RPLnull^\Psi$- derivation $\rho$ by appending variable elimination rules on some leaves of $\rho$ and (possibly) a variable introduction rule on the root (provided the sequents on the nodes are of an appropriate form). 
\end{definition}
\begin{examp}\label{RPLnull-zero-V}
The proof $\rho(X,Y,Z,n,m,V_{\phat}(X,n))$ for $n>\bar{0}$ in Example~\ref{ex.proofschema} is a $\RPLnull^\Psi V$-derivation.
\end{examp}
The link variables in an $\RPLnull^\Psi V$-derivation can be replaced by other $\RPLnull^\Psi V$-derivations:
\begin{definition}[proof composition]\label{def.proof-substitution}
Let $\rho_1$ be a $\RPLnull^\Psi V$-derivation with a root node $V_{\phat}(\vec{X},\vec{r})$ and let $\rho_2$ be a    
$\RPLnull^\Psi V$-derivation with (possibly several) leaf nodes $V_{\phat}(\vec{X},\vec{r})$ appearing at the set of positions 
$\Lambda$. Then  the composition of $\rho_1$ and 
$\rho_2$, denoted as $\rho_1 \circ \rho_2$, is defined as $\rho_2[\rho'_1]_\Lambda$ where $\rho'_1$ is the derivation of 
$\seq \phat(\vec{X},\vec{r})$, the premise of $V_{\phat}(\vec{X},\vec{r})$ (note that the last rule in $\rho_1$ is the variable introduction rule for  $V_{\phat}(\vec{X},\vec{r})$). $\rho_1$ and $\rho_2$ are called {\em composable} if 
there exists a proof variable $V$ which is the root node of $\rho_1$ and a leaf node of $\rho_2$.
\end{definition}
We did not write $\rho_2\{V_{\phat}(\vec{X},\vec{r}) \ass \rho'_1\}$ for $\rho_1 \circ \rho_2$ because we do not exclude that $V_{\phat}(\vec{X},\vec{r})$ is also the root node of $\rho_2$.
\begin{examp}\label{ex.proof-composition}
Let $\rho_2$ be the proof 
\[
\infer[V_{\phat}I]{V_{\phat}(X,p(n))}
 { \infer[R(\Theta(n,m))]{ \phat(X,p(n))}
    { \infer{P(X(n),\fhat(a,n)) \seq \phat(X,p(n))}
         { \infer{\seq \phat(X,p(n)), \neg P(X(n),\fhat(a,n))}
            { \infer{\seq \phat(X,p(n)) \lor \neg P(X(n),\fhat(a,n))}
               { \infer[V_{\phat}E]{\seq \phat(X,n)}{V_{\phat}(X,n)}
               }
             }
          }
    &
       \seq P(\fhat(Y(n),m),Z(n))
      }
}
\]
and $\rho_1$ be 
\[
\infer{V_{\phat}(X,n)}
{\infer{\seq \phat(X,n)}
  { \infer{\seq P(\fhat(Y(n),m),Z(n)) \land \phat(X,n)}
     { \seq \qhat(X,Y,Z,n,m)
      }
  }
}
\]
Then $\rho_1\circ \rho_2$ =
\[
\infer[V_{\phat}I]{V_{\phat}(X,p(n))}
 { \infer[R(\Theta(n,m))]{ \phat(X,p(n))}
    { \infer{P(X(n),\fhat(a,n)) \seq \phat(X,p(n))}
         { \infer{\seq \phat(X,p(n)), \neg P(X(n),\fhat(a,n))}
            { \infer{\seq \phat(X,p(n)) \lor \neg P(X(n),\fhat(a,n))}
               { \infer{\seq \phat(X,n)}
                  { \infer{\seq P(\fhat(Y(n),m),Z(n)) \land \phat(X,n)}
                    { \seq \qhat(X,Y,Z,n,m)
                   }
               }
               }
             }
          }
    &
       \seq P(\fhat(Y(n),m),Z(n))
      }
}
\]
\end{examp}
\begin{definition}[proof recursion]\label{def.proofrec}
Let $\rho(\bar{X},\vec{n},V_{\phat}(\vec{Y},\vec{m},k))$ be a proof with a leaf $V_{\phat}(\vec{Y},\vec{m},k)$ (occurring exactly once) and with the root $V_{\phat}(\vec{Y},\vec{m},p(k))$, where $\vec{Y}$ is a subvector of $\vec{X}$ and $(\vec{m},k)$ of $\vec{n}$ (if this is the case we say that $\rho$ admits proof recursion). We abbreviate 
$V_{\phat}(\vec{Y},\vec{m},k)$ by $V(k)$ and define 
\[
\begin{array}{l}
\If\ k= \bar{0}\ \Then\ \rhohat((\bar{X},\vec{n},V(k)) = V(\bar{0})\\[1ex]
\Else\ \rhohat((\bar{X},\vec{n},V(k)) =
\rho(\bar{X},\vec{n},V(k)) \circ \rhohat((\bar{X},\vec{n}\{k \ass p(k)\},V(p(k)))).
\end{array}
\] 
Note that from $V(\bar{0})$ we can finally derive $\phat(\vec{Y},\vec{m},\bar{0})$. We say that $\rhohat$ is the {\em inductive closure} of $\rho$.
\end{definition}
\begin{examp}\label{ex.inductive-closure}
Take $\rho(X,Y,Z,n,m,V_{\phat}(X,n))$ and $\rhohat(X,Y,Z,n,m,V_{\phat}(X,n))$ from Example~\ref{ex.proofschema}. Then $\rhohat$ is the inductive closure of $\rho$.
\end{examp}
Note that $\rhohat$ is obtained from $\rho$ by a kind of primitive recursion on proofs.
\begin{definition}[proof schema]\label{def.proof-schema}
We define proof schema inductively:
\begin{itemize}
\item Any $\RPLnull^\Psi V$-derivation is a proof schema.
\item If $\rho_1$ and $\rho_2$ are proof schemata and $\rho_1,\rho_2$ are composable then $\rho_1 \circ \rho_2$ is a proof schema.
\item If $\rho$ is a proof schema which admits proof recursion then the inductive closure of $\rho$ is a proof schema.
\item Let $\rho_1(\vec{X_1},\vec{n}),\ldots,\rho_\alpha(\vec{X}_\alpha,\vec{n})$ (for $\alpha >0$) be proof schemata over the parameter tuple $\vec{n}$  and let $\{C_1,\ldots,C_\alpha\}$ be conditions on the parameters in $\vec{n}$ which define a partition then 
\[
\begin{array}{l}
\If\ C_1 \ \Then\   \rho_1(\vec{X_1},\vec{n})\ \Else\\
\mbox{ } \If\  C_2 \ \Then\   \rho_2(\vec{X_2},\vec{n})\ \Else\\
\ldots\\
\mbox{   }\ \  \If\  C_{\alpha-1} \ \Then\   \rho_{\alpha-1}(\vec{X_{\alpha-1}},\vec{n})\ \Else\ \rho_\alpha(\vec{X}_\alpha,\vec{n})
\end{array}
\]
is a proof schema.
\end{itemize}
\end{definition}
\begin{examp}\label{ex.proofschema-2}
Consider the proofs $\rho$, $\rhohat$ and $\rho_0$ in Example~\ref{ex.proofschema}. All of them are proof  schemata; the proof schema $\rho_0$ is also a {\em refutation schema} of $\qhat(X,Y,Z,n,m)$, a concept which will be formally defined below.
\end{examp}
Our schematic proofs are proofs from sequents of the form $\seq F$, which we call schematic $F$-proofs. A schematic $F$-proof of $\seq$ is called a refutation schema of $F$. 
\begin{definition}[schematic $F$-proofs]\label{def.refutation-schema}
Let $F\colon \qhat(\vec{X},\vec{n})$ be the main schematic atom in a schematic definition $\Psi$. We define schematic $F$-proofs below. The $\RPLnull^\Psi V$-proof
\[
 \seq \qhat(\vec{X},\vec{n})
\]
is a schematic $F$-proof. 
\begin{itemize}
\item If $\rho_1$ and $\rho_2$ are schematic $F$-proofs of $S_1$ and $S_2$ from $\seq \qhat(\vec{X},\vec{n})$ and $\rho=$
\[
\infer[\xi]{S}{\deduce{S_1}{(\rho_1)}
                           &
                         \deduce{S_2}{(\rho_2)}
                        }
\]
for a binary rule $\xi$ then $\rho$ is a schematic $F$-proof of $S$.
\item If $\rho'$ is a schematic $F$-proof of $S'$ and $\rho=$
\[
\infer[\xi]{S}{\deduce{S'}{(\rho')}}
\]
for a unary rule $\xi$ then $\rho$ is a schematic $F$-proof of $S$.
\item Let $\rho_1$ be a schematic $F$-proof of $S$ where $S$ is of the form $V_{\phat}(\vec{Z},\vec{k})$. Let $\rho_2$ be a proof schema with one or several leaves $\lambda: V_{\phat}(\vec{Z},\vec{k})$ such that for all other leaves $\lambda$ of $\rho_2$ there are schematic $F$-proofs of ${\rm seq}(\lambda)$. Then $\rho_1 \circ \rho_2$  is a schematic $F$-proof of the end-sequent of $S_2$.
\item Let $\rho(\bar{X},\vec{n},V_{\phat}(\vec{Y},\vec{m},k))$ be a proof schema with a leaf $V_{\phat}(\vec{Y},\vec{m},k)$ and with the root $V_{\phat}(\vec{Y},\vec{m},p(k))$, where $\vec{Y}$ is a subvector of $\vec{X}$ and $(\vec{m},k)$ of $\vec{n}$ and assume that $\rho$ admits proof recursion, i.e. it is a proof schema of $V_{\phat}(\vec{Y},\vec{m},p(k))$ from $V_{\phat}(\vec{Y},\vec{m},p(k))$. Assume further that $\rhohat$ is defined as ($V(k)$ stands for $V_{\phat}(\vec{Y},\vec{m},k)$)
\[
\begin{array}{l}
\If\ k= \bar{0}\ \Then\ \rhohat((\bar{X},\vec{n},V(k)) = V(\bar{0})\\[1ex]
\Else\ \rhohat((\bar{X},\vec{n},V(k)) =
\rho(\bar{X},\vec{n},V(k)) \circ \rhohat((\bar{X},\vec{n}\{k \ass p(k)\},V(p(k)))).
\end{array}
\] 
If, for $l \leq k$, the proof schema $\rho(\bar{X},\vec{n}\{k \ass l\},V_{\phat}(\vec{Y},\vec{m},l))$ is a schematic $F$-proof (of its end-sequent) then $\rhohat((\bar{X},\vec{n},V(k))$ is a schematic $F$-proof of $V(\bar{0})$. 
\item  Let $\rho_1(\vec{X_1},\vec{n}),\ldots,\rho_\alpha(\vec{X}_\alpha,\vec{n})$ (for $\alpha >0$) be schematic $F$-proofs over the parameter tuple $\vec{n}$ and let $\{C_1,\ldots,C_\alpha\}$ be conditions on the parameters in $\vec{n}$ which define a partition; let $\rho(\vec{X_1},\vec{n})=$
\[
\begin{array}{l}
\If\ C_1 \ \Then\   \rho_1(\vec{X_1},\vec{n})\ \Else\\
\mbox{ } \If\  C_2 \ \Then\   \rho_2(\vec{X_2},\vec{n})\ \Else\\
\ldots\\
\mbox{   }\ \  \If\  C_{\alpha-1} \ \Then\   \rho_{\alpha-1}(\vec{X_{\alpha-1}},\vec{n})\ \Else\ \rho_\alpha(\vec{X}_\alpha,\vec{n})
\end{array}
\]
Then $\rho$ is a schematic $F$-proof of $S$.
\end{itemize}
A schematic $F$-proof of the empty sequent $\seq$ is called a refutation schema of $F$.
\end{definition}
\begin{examp}\label{ex.ref-schema}
The proof schema $\rho_0$ in Example~\ref{ex.proofschema} is a refutation schema of $\qhat(X,Y,Z,n,m)$.
\end{examp}
When the parameters in a refutation schema are instantiated with numerals, we obtain a $\RPLnull^\Psi V$ refutation.
\begin{theorem}\label{the.soundness-ref-schemata}
Let $\rho$ be a refutation schema of a schematic atom $\qhat(\vec{X},n_1,\ldots,n_\alpha)$. Then, for all numerals $\nu_1,\ldots,\nu_\alpha$, the evaluation of $\rho\{n_1 \ass \nu_1,\ldots,n_\alpha \ass \nu_\alpha\}$ is a $\RPLnull^\Psi V$ refutation of $\qhat(\vec{X},\nu_1,\ldots,\nu_\alpha)$.
\end{theorem}
\begin{proof}
Let $\rho' = \rho\{n_1 \ass \nu_1,\ldots,n_\alpha \ass \nu_\alpha\}$ for numerals $\nu_1,\ldots,\nu_\alpha$. 
If $\rho$ is defined only via proof composition then $\rho'$ is an $\RPLnull^\Psi V$ derivation of $\seq$ from axioms $\vdash \qhat(\vec{X},\nu_1,\ldots,\nu_\alpha)$ and thus an $\RPLnull^\Psi V$ refutation of $\qhat(\vec{X},\nu_1,\ldots,\nu_\alpha)$. 
If $\rho$ is defined via conditions $\{C_1,\ldots,C_\beta\}$ then exactly one of the conditions (let us say $C_i$) holds for the parameter assignment $\{n_1 \ass \nu_1,\ldots,n_\alpha \ass \nu_\alpha\}$. So $\rho' = \rho'_i$ and we may apply the induction hypothesis. 
\\[1ex]
The only nontrivial case occurs when $\rho$ is defined via proof recursion. Here we must proceed by induction on the number of proof recursions $\delta(\rho)$ in $\rho$. If $\delta(\rho) = 0$ we have no inductive closures and we are done. Otherwise assume that 
$\hat{\psi}(\vec{Y},\vec{m},k)$ is a minimal inductive closure in $\rho$ via the active parameter $k$, i.e. $\psi$ itself does not contain proof recursions. Assume that $\hat{\psi}(\vec{Y},\vec{m},k)$ is a schematic $F$-proof of a sequent $S$. We prove by induction that 
$\hat{\psi}(\vec{Y},\vec{m},k)$ evaluates to a derivation in $\RPLnull^\Psi V$. Under the parameter assignment $\sigma\colon \{n_1 \ass \nu_1,\ldots,n_\alpha \ass \nu_\alpha\}$ we obtain $\hat{\psi}(\vec{Y},\vec{\mu},\nu)$ for a vector of numerals $\vec{\mu}$ and a numeral $\nu$; $\hat{\psi}(\vec{Y},\vec{\mu},\nu)$ is a derivation of $\sigma(S)$. If $\nu = \bar{0}$ then by definition of inductive closure and by minimality of $\hat{\psi}$ $\hat{\psi}(\vec{Y},\vec{\mu},\bar{0})$ is a proof without proof recursion and we can argue as above. If $\nu > \bar{0}$ we have 
$$\hat{\psi}(\vec{Y},\vec{\mu},\nu) = \psi(\vec{Y},\vec{\mu},\nu) \circ \hat{\psi}(\vec{Y},\vec{\mu},p(\nu))$$
By induction hypothesis $\hat{\psi}(\vec{Y},\vec{\mu},p(\nu))$ evaluates to a derivation 
$\phi \in \RPLnull^\Psi V$ and so $\hat{\psi}(\vec{Y},\vec{\mu},\nu)$ evaluates to 
$$\psi' = \psi(\vec{Y},\vec{\mu},\nu) \circ \phi.$$
$\psi'$ is a composition of derivation without proof recursions and thus a derivation in $\RPLnull^\Psi V$. Replacing $\hat{\psi}(\vec{Y},\vec{\mu},\nu)$ in $\sigma(\rho)$ by $\psi'$ we obtain a derivation $\rho'$ of $\seq$ from $\qhat(\vec{X},\nu_1,\ldots,\nu_\alpha)$ with $\delta(\rho') < \delta(\rho)$. 
\end{proof}
Using the above formalism for refutation schemata, it is possible to extract a schematic structure representing the Herbrand sequent of the refutation.
\begin{examp}\label{ex.Herbrandschema}
Consider the refutation schema $\rho_0$, as defined in Example~\ref{ex.proofschema}. $\rho_0$ defines the s-substitution $\Theta(\bar{0},m) = \{X(\bar{0}) \leftarrow \fhat(Y(0),m), Z(0) \leftarrow \fhat(a,\bar{0})\}$. Moreover, as $\rho_0$ contains the recursive proof symbol $\hat{\rho}$ as axiom, we have to take the substitutions coming from $\hat{\rho}$ into account as well! 
In fact, $\Theta(\bar{0},m)$ will be applied to all the other s-substitutions coming from the rule applications ``above'', i.e. to all the substitutions in derivations corresponding to the proof variables in the leaves. 
$\hat{\rho}$ is recursively defined over the derivation $\rho$, which defines the s-substitution $\Theta(n,m) = \{X(n) \leftarrow \fhat(Y(n),m), Z(n) \leftarrow \fhat(a,n)\}$.
By construction, as $\hat{\rho}$ is recursive, $\Theta(p(n),m)$ is applied to $\Theta(n,m)$; to the result we apply $\Theta(p(p(n)),m)$ and so on. Intuitively, $\Theta^*(n,m)$ is is the sequence
$$\Theta(n,m)\Theta(p(n),m)\Theta(p(p(n)),m) \cdots \Theta(\bar{0},m).$$
or expressed recursively 
$$\Theta^*(n,m) =\{\Theta(0,m)\colon n=0,\ \Theta(n,m)\circ\Theta^*(p(n),m)\colon n>0\}.$$
\end{examp}
\begin{definition}[Herbrand system]
The Herbrand system of a schematic $F$-proof is defined inductively as:
\begin{itemize}
\item For $\seq F$ the Herbrand system is  $\{\emptyset\}$, the set containing the identical substitution.
\end{itemize}
\begin{itemize}
	\item Let $\rho$ be a schematic $F$-proof of the form
	\[
	\infer[\xi]{S}{\deduce{S_1}{(\rho_1)}
                      &
                      \deduce{S_2}{(\rho_2)}
                     }
	\]
	and assume the Herbrand system of $\rho_1$ and $\rho_2$ are $\Theta_1$ and $\Theta_2$. 
	If $\xi$ is a binary rule different to the resolution rule, then the Herbrand system of $\rho$ is the global s-unifier of $\Theta_1 \cup \Theta_2$, which can be computed after regularization of the proof. 
	If $\xi$ is a resolution rule of the form $R(\Theta)$, then prior regularization is mandatory and the Herbrand system is $(\Theta_1 \cup \Theta_2) \circ \Theta$.
	\item Let $\rho$ be a schematic $F$-proof of the form
	\[
	\infer[\xi]{S}{\deduce{S'}{(\rho')}}
	\]
	and assume that the Herbrand system of $\rho'$ is $\Theta$, then the Herbrand system of $\rho$ is $\Theta$.
	\item Let $\rho$ be of the form $\rho_1 \circ \rho_2$, then the Herbrand system of $\rho$ is $\Theta_1 \circ \Theta_2$, where $\Theta_1$ is the Herbrand system of $\rho_1$, and $\Theta_2$ is the Herbrand system of $\rho_2$, 
\item Let $\rho(\bar{X},\vec{n},V_{\phat}(\vec{Y},\vec{m},k))$ be a proof schema of $V_{\phat}(\vec{Y},\vec{m},p(k))$ from $V_{\phat}(\vec{Y},\vec{m},k)$, let $\Theta(\vec{n})$ be the global s-unifier of the derivation, and let $\rhohat$ be defined as ($V(k)$ stands for $V_{\phat}(\vec{Y},\vec{m},k)$)
\[
\begin{array}{l}
\If\ k= \bar{0}\ \Then\ \rhohat((\bar{X},\vec{n},V(k)) = V(\bar{0})\\[1ex]
\Else\ \rhohat((\bar{X},\vec{n},V(k)) =
\rho(\bar{X},\vec{n},V(k)) \circ \rhohat((\bar{X},\vec{n}\{k \ass p(k)\},V(p(k)))).
\end{array}
\] 
To ensure that $\rhohat$ is regular we must ensure that all variable expressions in $\rho$ contain the parameter $n$. If this is not the case the variables need to be renamed.
Then the Herbrand system of $\rho$ is defined as
\[
\begin{array}{l}
\If\ k= \bar{0}\ \Then\ \Theta^*(\vec{n}) = \Theta(\vec{n})\\[1ex]
\Else\ \Theta^*(\vec{n}) =
\Theta(\vec{n}) \circ \Theta^*(\vec{n}\{k \leftarrow p(k)\}).\\[1ex]
\end{array}
\] 
\item  Let $\rho(\vec{X_1},\vec{n})=$
\[
\begin{array}{l}
\If\ C_1 \ \Then\   \rho_1(\vec{X_1},\vec{n})\ \Else\\
\mbox{ } \If\  C_2 \ \Then\   \rho_2(\vec{X_2},\vec{n})\ \Else\\
\ldots\\
\mbox{   }\ \  \If\  C_{\alpha-1} \ \Then\   \rho_{\alpha-1}(\vec{X_{\alpha-1}},\vec{n})\ \Else\ \rho_\alpha(\vec{X}_\alpha,\vec{n})
\end{array}
\]
be a schematic $F$-proof, and assume that the Herband systems of $\rho_1, \ldots, \rho_\alpha$ are $\Theta_1, \ldots, \Theta_\alpha$. Then the Herbrand system of $\rho$ is
\[
\begin{array}{l}
\If\ C_1 \ \Then\   \Theta_1\ \Else\\
\mbox{ } \If\  C_2 \ \Then\   \Theta_2\ \Else\\
\ldots\\
\mbox{   }\ \  \If\  C_{\alpha-1} \ \Then\  \Theta_{\alpha-1}\ \Else\ \Theta_\alpha
\end{array}
\]
\end{itemize}
\end{definition}
\begin{examp}
Consider the Herbrand system from Example \ref{ex.Herbrandschema} and the fixed parameters $n = 1$ and $m=0$. Then,
\begin{eqnarray*}
\Theta^*(1,0) &=& \Theta(1,0)\Theta^*(p(1),0)\\
                         &=&  \{X(1) \ass \fhat(Y(1),0), Z(1) \ass \fhat(a,1)\}\{X(0) \ass \fhat(Y(0),0), Z(0) \ass \fhat(a,0)\}\\
&=& \{\{X(1) \ass Y(1),\ Z(1) \ass f(a),\ X(0) \ass Y(0),\ Z(0) \ass a\}.
\end{eqnarray*}
Applying $\Theta^*(1,0)$ to the initial sequents results in an unsatisfiable set of sequents.
\end{examp}
\begin{theorem}\label{sound-Herbrand-systems}
Let $\Theta$ be a Herbrand system of a refutation schema $\rho$ of $\qhat(\vec{X},n_1,\ldots,n_\alpha)$. Then, for all numerals 
$\nu_1,\ldots,\nu_\alpha$, the set of sequents 
$$\{\vdash \qhat(\vec{X},n_1,\ldots,n_\alpha)\Theta\}\{n_1 \ass \nu_1,\ldots,n_\alpha \ass \nu_\alpha\}$$
is unsatisfiable.
\end{theorem}
\begin{proof}
For all numerals, first apply the s-substitutions defined in $\Theta$ to $\rho$. The initial sequents are then instances of $\vdash \qhat(\vec{X},n_1,\ldots,n_\alpha)$, and the resolution rules turn into applications of cut rules. 
The thus obtained derivation is in $\RPLnull^\Psi$. As this calculus is sound the set of sequents $\{\vdash \qhat(\vec{X},n_1,\ldots,n_\alpha)\Theta\}\{n_1 \ass \nu_1,\ldots,n_\alpha \ass \nu_\alpha\}$ occurring at the leaves is unsatisfiable.
\end{proof}
The formalism for schematic proofs presented in this paper is more powerful than that defined in~\cite{Thesis.Lolic.2020} and in~\cite{CLL.2021} as these former approaches missed the inductive closure of proofs as a syntactic proof object. In fact inductive closures could only be computed for given parameter assignments, a representation on the syntax was missing. As a consequence there was no way to define Herbrand systems in a general way. Moreover, the proof recursion defines new schematic proofs which can be called from others and thus considerably extends the expressivity of the former approaches.     

\section{Conclusion}
In~\cite{CLL.2021} and~\cite{Thesis.Lolic.2020} the concept of $s$-substitution (schematic substitution) was defined which is crucial to the development of a schematic resolution calculus and thus to the schematic CERES method. However, no algorithms were defined to concatenate and apply $s$-substitutions in presence of variable expressions and extended schematic terms. In this paper we fill this hole and define a new schematic substitution calculus for a class of extended schematic terms we call {\em standard}. While schematic unification was defined in~\cite{CLL.2021} and~\cite{Thesis.Lolic.2020}, no unification algorithm was given. Here we define a sound (though incomplete) unification algorithm for sets of standard terms which substantially extends the practical applicability of schematic refutational calculi.  
We introduced the new calculus $\RPLnull^\Psi V$ for the construction of schematic refutations of quantifier-free formula schemata and applied this calculus to the definition of Herbrand systems corresponding to the schematic refutations. We applied this calculus to a  formula schema originating from the proof analysis method CERES, where it formalizes the derivations of the cut formulas in the original proof schema $\Phi$. From the refutation schema of this formula schema, we constructed a Herbrand system, which can be used to compute the Herbrand system of the original proof schema $\Phi$ (the final step in schematic cut-elimination). The formalism for proof schemata presented in this work is new, extends the expressivity of former approaches, and simplifies existing notions of proof as schema. 
We think that this new formalism improves the specification of inductive refutations considerably and thus the potential of interactive proof analysis. 
\bibliographystyle{plain}
\bibliography{references}

\end{document}